\documentclass[reqno]{amsart}

\usepackage{hyperref}
\usepackage{amsmath}
\usepackage{amssymb}
\usepackage{amsfonts}
\usepackage{amsthm}
\usepackage{latexsym}
\usepackage{esint}
\usepackage{mathtools}
\usepackage{mathrsfs}
\usepackage{graphicx}
\usepackage{subcaption}
\usepackage{wrapfig}
\usepackage{bbm}
\usepackage{color}
\usepackage{bm}
\usepackage{float}

\usepackage{tikz}
\usepackage{tikz-3dplot}

\usepackage{tikz-cd}
\usetikzlibrary{lindenmayersystems}
\usetikzlibrary{decorations.pathreplacing}

\usepackage[top=1in, bottom=1.25in, left=1.10in, right=1.10in]{geometry}
\usepackage{todonotes}
\usepackage{stackengine}
\allowdisplaybreaks

\newtheorem{theorem}{Theorem}[section]
\newtheorem{lemma}[theorem]{Lemma}

\newtheorem{corollary}[theorem]{Corollary}

\theoremstyle{definition}
\newtheorem{definition}[theorem]{Definition}

\theoremstyle{remark}
\newtheorem{remark}[theorem]{Remark}

\numberwithin{equation}{section}

\newcommand{\bx}{\mathbf{x}}
\newcommand{\by}{\mathbf{y}}

\newcommand{\bk}{\mathbf{k}}

\newcommand{\bn}{\mathbf{n}}
\newcommand{\be}{\mathbf{e}}

\newcommand{\dy}{\, \mathrm{d}\mathbf{y}}
\newcommand{\dd}{\,\mathrm{d}}

\newcommand{\dt}{\, \mathrm{d}t}
\newcommand{\ds}{\, \mathrm{d}s}

\newcommand{\naby}{\nabla_{\mathbf{y}}}

\newcommand{\Dely}{\Delta_{\mathbf{y
}}}

\newcommand{\R}{\mathbb{R}}

\begin{document}

\title{Stochastically-constrained Koiter shell models.}



\author{Prince Romeo Mensah}
\address{Faculty of Mathematics, University of Duisburg-Essen, Thea-Leymann-Strasse 9, 45127 Essen,
Germany}
%
\author{Pierre Marie Ngougoue Ngougoue}
\address{Faculty of Mathematics, University of Duisburg-Essen, Thea-Leymann-Strasse 9, 45127 Essen,
Germany}

\date{\today}

\begin{abstract}
We derive stochastically-constrained Koiter shell models in line with the SALT (Stochastic Advection by Lie Transport)  approach introduced by Holm [Proc. A. 471 (2015)].  First, we deduce the stochastic partial differential equations for the generalised nonlinear elastic and linear elastic Koiter shell models with  abstract functional derivatives of their corresponding membrane and flexural energies. We then present a prototype for a stochastically-constrained (simplified)  linearised Koiter shell model
that captures stiffness effects arising from shell curvature, bending and membrane stresses, interior and surface forces, and, more generally, stochastic buckling.
Finally, we show that if a weak pathwise solution of this prototype is parametrised by a suitably chosen family of noise coefficients, we obtain in the parameter limit, the deterministic viscoelastic shell model with viscous damping. 

\end{abstract}



\maketitle
 
\section{Introduction}  

The (stationary) Koiter shell model is a two-dimensional model obtained as a three-dimensional model reduction of elasticity under two main assumptions. Firstly, the shell is assumed to be very thin so that all stresses are planar and parallel to the mid-surface. The second is the Kirchhoff-Love assumption that requires normals to remain straight, normal, and unstretched. Since Koiter initial work, several attempts were made to give a  rigorous justification for the model reduction until this was finally achieved using Gamma-convergence by  Le Dret \& Raoult \cite{dret1996membrane} and Friesecke, James, Mora \& M\"uller \cite{friesecke2003derivation}. We refer to Ciarlet's book \cite{ciarlet2005introduction} for a thorough overview on elasticity and Koiter shell models.

Building on the deterministic reduced model,   our first goal, performed in Sections \ref{sec:modelling}  and  \ref{sec:proto},  is to rigorously incorporate stochasticity into the Koiter shell models. Incorporating stochasticity into shell models is particularly useful for studying stochastic buckling \cite{papadopoulos2009buckling, reuter2025probabilistic, schafer2006stochastic}. In (deterministic) structural mechanics, buckling is analysed under the assumption that certain crucial information like critical load is known. In practise, however, there are several structural imperfections in shells such as random loads, geometric imperfections, manufacturing tolerance and material variability (e.g. shell thickness), see Figure \ref{fig.tubular}. Since shells are extremely sensitive to tiny defects and imperfections, adding stochasticity allows engineers to analyse possible structural failures  under random fluctuations. That being said, how stochasticity is added is also important. Rather than merely appending noise to the equivalent deterministic equation, adding stochasticity to the variational principle ensures a physical consistency between energy fluctuations and constraints of the shell's geometry. 
\begin{figure}
\begin{center}
\begin{tikzpicture}[scale=2]
    \begin{scope}     
\draw [dotted] (0.5,0.4) ellipse [x radius=0.1cm, y radius=0.23cm];
 
\draw [dotted] (1.15,0.4) ellipse [x radius=0.1cm, y radius=0.23cm];

    \node at (1.15,0.4) {$\bullet$};
    \node at (1.15,0.3) {$\bx$}; 
 
      \draw [thick] (0.5,0.63) .. controls (0.9,0.83) and (1.2,0.53) .. (2,0.43);  
     
          \draw [thick] (0.5,0.17) .. controls (0.9,0.37) and (1.2,0.07) .. (2,-0.03); 

\draw[dotted] (2,0.2)ellipse [x radius=0.1cm, y radius=0.23cm];
 
  \end{scope}

*****MIDDLE IMAGE*************

\begin{scope}
      \draw [thick] (3,0.7) to (3,0.07);  
      \draw [dotted] (3.15,0.7) ellipse [x radius=0.14cm, y radius=0.05cm];
            \draw [dotted] (3.15,0.07) ellipse [x radius=0.14cm, y radius=0.05cm];
       \draw [thin] (3.3,0.7) to (3.3,0.07);  
\end{scope}

*****RIGHT IMAGE*************

*****Thick front curved lines*************
\draw[thick] (4,0.07) to[bend left] (5,0.069); 

*****back curved lines*************
\draw[thick] (4.2,0.27) to[bend left] (5.2,0.269); 

*****left sided lines*************
\draw[thick] (4,0.07) to (4.2,0.27); 

*****right sided lines*************
\draw[thick] (5.2,0.269) to (5,0.069); 

*****labels and arrows************* 
    \draw [thin, ->] (4.5,0.7) -- (4.5,0.5);
    \draw [thin, ->] (4.4,0.65) -- (4.4,0.45);
    \draw [thin, ->] (4.3,0.60) -- (4.3,0.40);
    
    \draw [thin, ->] (4.8,0.7) -- (4.8,0.5);
    \draw [thin, ->] (4.7,0.65) -- (4.7,0.45);
    \draw [thin, ->] (4.6,0.60) -- (4.6,0.40);
    
    \draw [thin, ->] (5.1,0.7) -- (5.1,0.5);
    \draw [thin, ->] (5,0.65) -- (5,0.45);
    \draw [thin, ->] (4.9,0.60) -- (4.9,0.40); 
\end{tikzpicture}
      \vspace{-10pt}
  \caption{Left: A thin elastic cylinder (e.g. blood vessel) with an imperfection characterised by a slight deviation in the diameter at point $\bx$.
Middle: A thin cylinder with varying shell thickness from right to left.
   Right: A  thin shell subject to load.}
    \vspace{-5pt}
      \label{fig.tubular}
\end{center}
\end{figure}
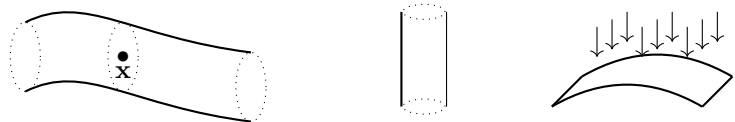
%
The SALT (Stochastic Advection by Lie Transport) approach by Holm \cite{holm2015variational} achieves such physical consistency and has since been successfully implemented to achieve the stochastic  Euler and Navier--Stokes equations \cite{cotter2019numerically, drivas2020circulation} with transport noise. It's versatility is further illustrated by its extension to soliton dynamics \cite{holm2016variational}, to geophysical fluids \cite{cotter2020data, crisan2023theoretical} and to wave-current interactions \cite{holm2021stochastic, holm2021stochastic1}, amongst others.   In light of this, we extend its application to Koiter shells with the caveat of implementing SALT without explicitly resorting to Lie Algebra. As it turns out, the equivalent use of purely variational calculus suffices,
 at least for our need. This is the object of the next section.

  Recent works \cite{flandoli2021scaling, flandoli2020convergence, flandoli2021high,  galeati2020convergence} have demonstrated the regularising effects of transport noise in fluids. 
Our second goal  will be to demonstrate that a similar result can be applied to the linearised elastic material under study.  To better explain our approach, however, we first give a summary of \cite{flandoli2021scaling} which  will be our main point of reference. Of particular interest, this reference deals with the construction of weak solutions for the deterministic Navier--Stokes equations from the Stochastic Euler equations with transport noise.

To achieve their desired result, the authors  first construct a weak martingale solution for the Stochastic Euler equation with transport noise (which, for brevity, we henceforth refer to simply as the ``SPDE" throughout this review) by passing through the finite-dimensional   SDE \cite[(3.2)]{flandoli2021scaling}, equipped  with a suitably designed family of noise coefficients. We remark, however, that this SDE is not necessarily a projection of the expected Stochastic Euler equations with transport noise, since there are many finite-dimensional projections within the convective terms. In the present work, we do not pursue this step, since we can already infer the existence of our desired continuum solution for the linear shell equations from the variational analysis developed in the next two sections.

The second step in \cite{flandoli2021scaling} consists in considering a family of solutions to the SPDE and proving their convergence  to the deterministic Navier--Stokes equations. Since this family is infinite-dimensional, they authors  must again rely on the stochastic compactness machinery employed  in the first step (see, in particular,  the discussion at the  bottom of \cite[page 582]{flandoli2021scaling}).

The approach pursued here is substantially different and,  owing to  the linear structure of the shell, considerably simpler.
Indeed, we show that for a suitably chosen finite-dimensional family of transport noise and corresponding  finite-dimensional projections of the dataset, there exists a solution to an SDE that preserves the same transport structure as the continuum problem.
Although this SDE is not necessarily obtained as a direct projection of the linear shell equation, its solution evolves within the same finite-dimensional space onto which the data are projected, and both the discrete and continuum systems retain identical  transport structure. Moreover,  as noted earlier, the data of the SDE are precisely the finite-dimensional projections  of the data for the continuum equation. Our interest is therefore not to pass first from the SDE to an  SPDE and subsequently  to the PDE, but rather to pass directly from the finite-dimesnional SDE to the deterministic viscoelastic shell equation. With this approach, one benefits from standard theorems for (linear) SDE and thereby avoid the  stochastic compactness machinery.

\section{Modelling}
\label{sec:modelling} 
 
We let $\Omega\subset \mathbb{R}^3$ be an open, bounded, nonempty and connected reference domain  with an elastic shell $\partial\Omega:=\Gamma \times (-\epsilon_0,\epsilon_0) \subset \mathbb{R}^3$  of arbitrarily small thickness $2 \epsilon_0>0$ and a middle surface  $\Gamma \subset \mathbb{R}^{2}$. For simplicity, we identify $\Gamma$ with the flat torus endowed with periodic boundary condition. The same result, however, hold for any general mid-surface that is clamped at its boundary. 
 Now, we suppose that the shell is parametrised by a $C^3(\Gamma;\mathbb{R}^3)$-injective mapping
\begin{align*}
\bm{\varphi}:\Gamma \rightarrow\mathbb{R}^{3},
\qquad
\bm{\varphi}(\by):=(\varphi_1(\by),\varphi_2(\by),\varphi_3(\by))^\intercal
\end{align*}
such that $\partial_{y_1} \bm{\varphi}$ and $\partial_{y_2} \bm{\varphi}$ are linearly independent. Then, this pair form a covariant basis of the tangent plane to the middle surface $\bm{\varphi}(\Gamma )$  at the point $\bm{\varphi}(\by)$. Also, the function
\begin{align}
\label{nby}
\bn:\Gamma \rightarrow\mathbb{R}^{3},
\qquad
\mathbf{n}(\by)=\frac{\partial_{y_1} \bm{\varphi}  \times \partial_{y_2} \bm{\varphi} }{\vert \partial_{y_1} \bm{\varphi}  \times \partial_{y_2} \bm{\varphi}  \vert},
\end{align}
is a well-defined unit vector  normal to the surface $\bm{\varphi}(\Gamma )$ at $\bm{\varphi}(\by)$.    
All together, the three vectors 
\begin{align*}
\mathcal{B}:=
\{\partial_{y_1} \bm{\varphi} , \partial_{y_2} \bm{\varphi} ,\bn \}
\end{align*}
make up the covariant basis at the point $\bm{\varphi}(\by)$. 
Furthermore, from $\mathcal{B}$, we obtain a contravariant basis
\begin{align*}
\mathcal{B}^\star:=
\{\partial_{y_1}^\star \bm{\varphi} , \partial_{y_2}^\star \bm{\varphi} , \bn^\star  \}
\end{align*} 
at $\bm{\varphi}(\by)$ by defining elements $\be^*\in \mathcal{B}^\star$ to be vectors such that $\be^*\cdot\be=\delta_{ij}$ for each $\be\in \mathcal{B}$. In particular, we have
\begin{align*} 
\partial_{y_1}^\star \bm{\varphi} =\frac{\partial_{y_2} \bm{\varphi}   \times \bn }{\vert \partial_{y_1} \bm{\varphi} \times \partial_{y_2} \bm{\varphi} \vert}, 
\qquad 
\partial_{y_2}^\star \bm{\varphi} =\frac{-(\partial_{y_1} \bm{\varphi}  \times \bn )}{\vert \partial_{y_1} \bm{\varphi}  \times \partial_{y_2} \bm{\varphi} \vert},\qquad
\bn^\star =\bn,
\end{align*}
and  the vectors $\partial_{y_1}^\star \bm{\varphi} $ and $\partial_{y_2}^\star \bm{\varphi} $ are also in the tangent plane to the middle surface $\bm{\varphi}(\Gamma )$  at  $\bm{\varphi}(\by)$.
\mbox{Additionally,}
if $\dy$ is the area element along $\Gamma $, then
\begin{align}
\label{areaElem}
\mathrm{d}\by_{\mathbf{n}}:=\vert \partial_{y_1} \bm{\varphi}  \times \partial_{y_2} \bm{\varphi}  \vert \mathrm{d}\by 
\end{align}
is the area element along  $\bm{\varphi}(\Gamma )$. Here, note that  
\begin{align*} 
\vert \partial_{y_1} \bm{\varphi} \times \partial_{y_2} \bm{\varphi}  \vert\neq0
\end{align*} 
since the elements in $\mathcal{B}$ are linearly independent.

To give a well-defined notion of distance on the surface $\bm{\varphi}(\Gamma )$, we introduce its  metric tensor $\mathbb{A}$  whose covariant and contravariant components are
\begin{align*}
A^{ij}:=  \partial_{y_i} \bm{\varphi} \cdot \partial_{y_j} \bm{\varphi} 
\quad \text{ and }\quad {A^{ij}}^\star:= \partial_{y_i}^\star \bm{\varphi} \cdot \partial_{y_j}^\star \bm{\varphi}, \qquad i,j=1,2 ,
\end{align*}
respectively. The metric tensor $\mathbb{A}$ is symmetric and positive definite as the scalar product is commutative, and the elements in $\mathcal{B}$ (and in $\mathcal{B}^\star$) are linearly independent, respectively. Next, for a well-defined notion of curvature on the surface $\bm{\varphi}(\Gamma )$, we introduce its curvature tensor $\mathbb{B}$  whose covariant  components are
\begin{align}
\label{secFunForm}
B^{ij}:=\bn \cdot\partial_{y_iy_j}^2\bm{\varphi}, \qquad i,j=1,2.
\end{align}
The curvature tensor is also symmetric since  second derivatives of a smooth function commute; however, unlike the metric tensor $\mathbb{A}$,  it is not necessarily  positive definite.

At this point, we have all the required information to describe our reference thin elastic shell. We now proceed to model its time evolution by introducing dynamics. For this purpose, 
we assume that for each time $t\in\overline{I}:=\overline{(0,T)}$ where $T>0$ is fixed, the shell  deforms along the normal direction according to a displacement field $\eta \mathbf{n} : I \times \Gamma  \rightarrow\R^3$,   which is subject to random fluctuations (noise).
While deterministic elastic materials do not exhibit advection (or transport) since they do not ``flow" like fluids, 
advection can nevertheless arise in the stochastic setting as a way to model spatially correlated noise. In particular, imperfections such as material inhomogeneities, or fluctuating loads are typically not independent pointwise, but exhibit correlations along the surface. These correlations can be modelled by prescribing vectors fields along which the stochastic perturbations act.
 In this sense, the advection operator should not be interpreted as  physical transport of material, but rather as a mechanism to introduce geometric-consistent noise. 
Consequently, we model the evolution of  the displacement by a Stratonovich stochastic differential equation given by
\begin{align}\label{shellVelocity}
\dd\eta(t,\by) =\dot{\eta}(t,\by) \dd t+\frac{1}{2}\sum_{i=1}^N\bm{\sigma}_i(\by)\cdot\naby \eta (t,\by)\circ\dd W_t^i,
\end{align}
where $N\in\mathbb{N}$,  $(\bm{\sigma}_i)_{i=1}^N$ is a prescribed family of time-independent vector fields   (deterministic or random, and possibly  obtained from data), and $(W_t^i)_{i=1}^N$ is a family of  independent, identically distributed Brownian motions.
Moreover,
\begin{align*}
\dot{\eta}:=\partial_t\eta
\end{align*}
denotes the deterministic velocity (drift) of the shell.  Here, we remark that  the factor $\tfrac{1}{2}$ in the noise is just a convenient choice that can be omitted without any loss of generality. Note that  the actual  displacement is $\eta\bn$, not $\eta$; therefore, the corresponding  shell's velocity is $\partial_t\eta\bn$. For convenience, however, we shall always omit the unit normal when referring  to physical quantities such as the displacement, velocity and the soon-to-be introduced momemtum.  Furthermore,  observe that in the absence of noise, that is,  $\bm{\sigma}_i\equiv \bm{0}$,  \eqref{shellVelocity} reduces to the physical deterministic phenomenon, in which the total and partial time derivatives of the shell's displacement coincide due to the lack of advection.  
Finally, as is standard in stochastic analysis,  the unknown $\eta$ of the stochastic differential equation \eqref{shellVelocity}
is understood as  a stochastic process $(\omega,t,\by)\mapsto \eta(\omega,t,\by).$ The dependence on the outcome $\omega\in \Sigma$ in a probability space $(\Sigma,\mathcal{F},\mathbb{P})$ is omitted, and the unknown is  defined in the following integral sense
\begin{align*} 
\eta(t,\by) =\eta(0,\by) +\int_0^t\dot{\eta}(s,\by) \dd s+\frac{1}{2}\sum_{i=1}^N\int_0^t\bm{\sigma}_i(\by)\cdot\naby \eta (s,\by)\circ\dd W_s^i 
\end{align*}
\begin{figure}[H] 
    \centering
    \begin{subfigure}[b]{0.25\textwidth}
        \includegraphics[width=\textwidth]{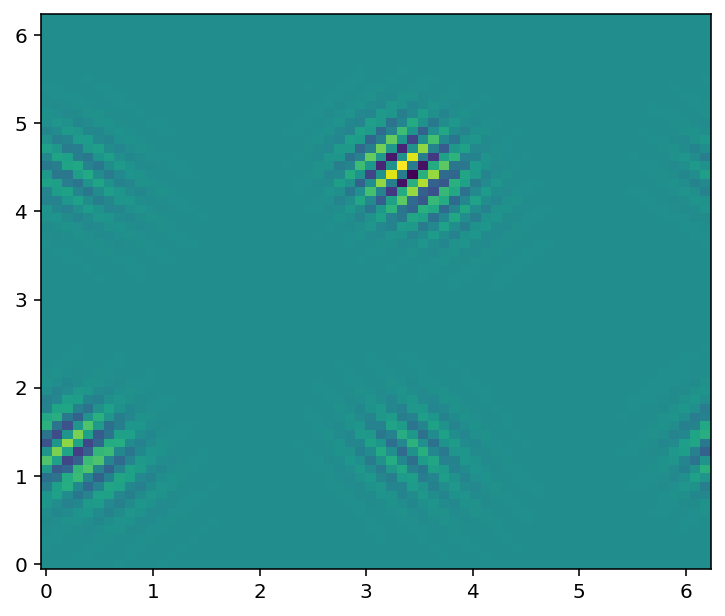}
    \end{subfigure}
    ~ 
    \begin{subfigure}[b]{0.25\textwidth}
        \includegraphics[width=\textwidth]{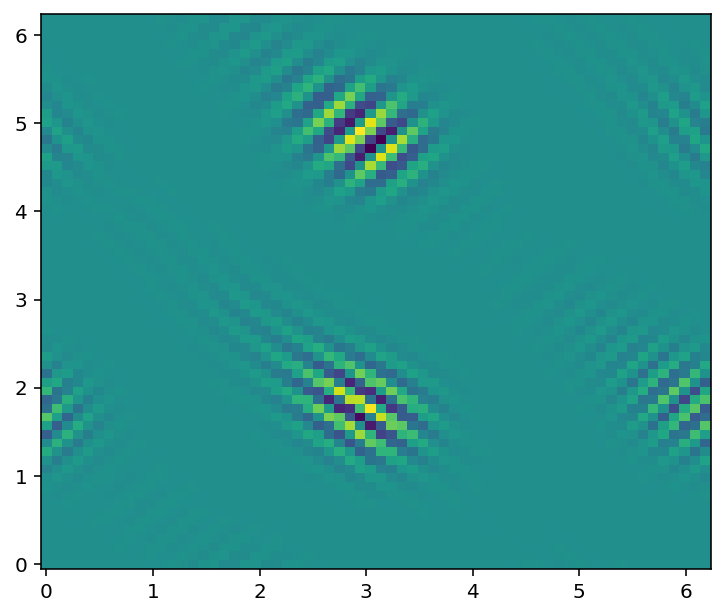}
    \end{subfigure}
    ~ 
    \begin{subfigure}[b]{0.25\textwidth}
        \includegraphics[width=\textwidth]{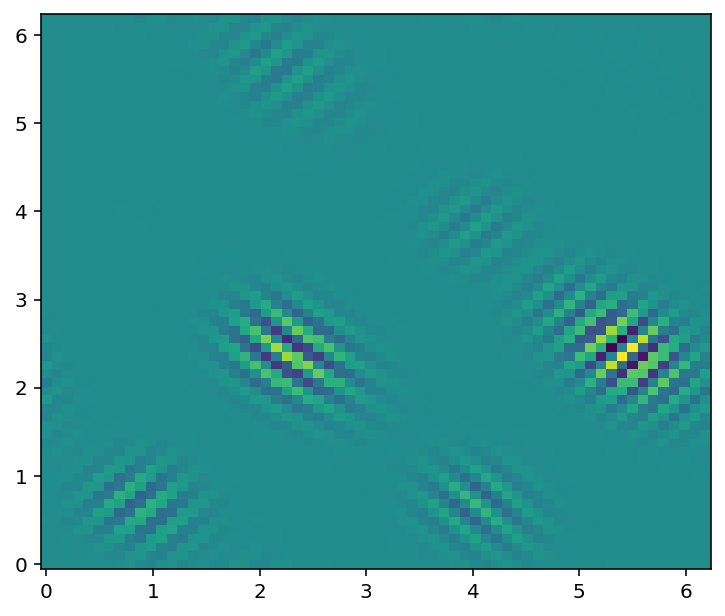}
    \end{subfigure}
    \caption{Three snapshots for $\eta(T,\by)$ solving \eqref{shellVelocity} for $N = 2$. The stochastic term is defined by two vector fields  $\bm{\sigma}_1=2(\sin(y_1),-\cos(y_2))^\top$, $\bm{\sigma}_2=2(-\cos(y_1),\sin(y_2))^\top $ on the torus $\Gamma=[-2\pi,2\pi]^2$ with initial condition $\eta(0,\by)=\exp(-((y_1-\pi)^2+(y_2-\pi)^2))$.}\label{fig:3pix}
\end{figure}
for all $t \in  \overline{I}$ (see e.g., Figure \ref{fig:3pix}).
We may now parametrise the deformed boundary according to the mapping
\begin{align}
\label{deformEta}
\bm{\varphi}_\eta:\overline{I}\times \Gamma \rightarrow\mathbb{R}^{3},
\qquad
\bm{\varphi}_{\eta(t)}(\by)=\bm{\varphi}(\by) + \mathbf{n}(\by)\eta(t, \by), 
\end{align}
resulting in the deformed middle surface  
$\bm{\varphi}_{\eta(t)}(\Gamma )$. More generally, 
\begin{align*}
\partial{\Omega_{\eta(t)}}=\big\{\bm{\varphi}_{\eta(t)}(\by):=\bm{\varphi}(\by)+\mathbf{n}(\by)\eta(t,\by)\, :\, t\in I, \by\in \Gamma \big\}
\end{align*}
represents the boundary of the flexible domain at any instant of time $t\in I$. 

Under this deformation,  the metric tensor $\mathbb{A}_\eta $ of the deformed middle surface  $ \bm{\varphi}_{\eta(t)}( \Gamma )$ has covariant components given by
\begin{align*} 
A_\eta^{ij}=&
\partial_{y_i} \bm{\varphi}_{\eta} \cdot \partial_{y_j} \bm{\varphi}_{\eta}, \qquad i,j=1,2.
\end{align*} 
By subtracting the covariant components of the metric tensor  of the original middle surface from those of the deformed surface and taking half of the difference, we obtain
\begin{align*}
G_\eta^{ij} :=& \frac{1}{2}(A^{ij}_\eta-A^{ij}), \qquad i,j=1,2
\end{align*}
which are the covariant components of the  change of metric tensor $\mathbb{G}_\eta$, measuring the variation of the metric from the   surface $\bm{\varphi}(\Gamma )$ to the deformed  surface  $ \bm{\varphi}_{\eta(t)}( \Gamma )$.

With this information, we now wish to find elements of the corresponding curvature tensor which encode information about the second derivatives of $ \bm{\varphi}_{\eta(t)}( \Gamma )$.
To  this end, we first observe  that,  since the vectors in $\mathcal{B}$ are linearly independent on $\Gamma $ and $\eta$ is sufficiently smooth, the vectors
\begin{align*}
\mathcal{B}_\eta=\{\partial_{y_1} \bm{\varphi}_{\eta}, \partial_{y_2} \bm{\varphi}_{\eta}, \mathbf{n}_{\eta}
\},
\end{align*}
where
\begin{align*}
\mathbf{n}_\eta:\overline{I}\times \Gamma \rightarrow\mathbb{R}^{3},
\qquad
\mathbf{n}_{\eta(t)}(\by)
=& \partial_{y_1} \bm{\varphi}_{\eta}  \times \partial_{y_2} \bm{\varphi}_{\eta},
\end{align*}
remain linearly independent on $\Gamma $,  provided $\eta$ is sufficiently small, for instance in the 
$C^1(\Gamma  )$-norm or, more generally,   in the $W^{1,\infty}(\Gamma  )$-norm; that is $\Vert\naby \eta\Vert_{L^\infty(\Gamma)}\leq L$ for some $L>0$.

The curvature tensor $\mathbb{B}_\eta $ of  the deformed middle surface  $ \bm{\varphi}_{\eta(t)}( \Gamma )$ then has entries given by 
\begin{align*}
B_\eta^{ij} 
=& \mathbf{n}_{\eta}  \cdot
\partial_{y_iy_j}^2 \bm{\varphi}_{\eta}, \qquad i,j=1,2. 
\end{align*}
Accordingly, given the curvature tensor \eqref{secFunForm} of the original middle surface $\bm{\varphi}(\Gamma )$, we obtain
\begin{align*}
{R_\eta^{ij}}^\sharp  :=&\frac{B_\eta^{ij}}{\vert \partial_{y_1} \bm{\varphi} \times \partial_{y_2} \bm{\varphi}  \vert}-B^{ij}, \qquad i,j=1,2,
\end{align*}
which are the  covariant components of the  \textit{modified} change of curvature tensor $\mathbb{R}_\eta^\sharp$,  representing the change in curvature induced by the deformation $\eta.$
%
%
%
The use of ${R_\eta^{ij}}^\sharp$, rather than the expected covariant components $R_\eta^{ij}$ of the \textit{exact} change of curvature tensor
\begin{align*}
R_\eta^{ij}  :=&\widehat{\mathbf{n}}_{\eta}  \cdot
\partial_{y_iy_j}^2 \bm{\varphi}_{\eta}-B^{ij}, \qquad i,j=1,2
\quad\text{ where }\quad
\widehat{\mathbf{n}}_{\eta} 
=
\frac{\partial_{y_1} \bm{\varphi}_{\eta}  \times \partial_{y_2} \bm{\varphi}_{\eta}}{\vert \partial_{y_1} \bm{\varphi}_{\eta}  \times \partial_{y_2} \bm{\varphi}_{\eta}\vert},
\end{align*}
is due to Roquefort \cite{roquefort2001quelques}, originally proposed by Ciarlet \cite{ciarlet2000modele} and already  appearing in Koiter's seminal work \cite{koiter1966nonlinear}. Indeed, in  a physical setting where no bound is imposed on the shell displacement, the deformed basis vectors $\partial_{y_1} \bm{\varphi}_{\eta}$ and $\partial_{y_2} \bm{\varphi}_{\eta}$ belonging to $\mathcal{B}_\eta,$ may become linearly dependent, even though the corresponding reference vectors $\partial_{y_1} \bm{\varphi}$ and $ \partial_{y_2} \bm{\varphi}$ are linearly independent. In such a case, the  covariant components $R_\eta^{ij}$ of the exact change of curvature tensor is ill-defined,  since the denominator in the normalised normal vector $\widehat{\mathbf{n}}_{\eta}$ vanishes. This motivates the introduction of ${R_\eta^{ij}}^\sharp$, which  prevents such  degeneracy \cite{ciarlet2001justification}.

\subsection{The nonlinear shell model}
To described the elastic energy associated with the deformation of the shell described above, we first introduce
two fourth-order tensors $\mathbb{C}_e=( C_e^{ijkl})_{i,j,k,l =1}^2$  corresponding to the shell elasticity tensor. Its entries
\begin{align*}
C_e^{ijkl}
:=& 
\frac{4\lambda_e\mu_e}{\lambda_e+2\mu_e}(\partial^\star_{y_i}\bm{\varphi} \cdot\partial^\star_{y_j}\bm{\varphi} )(\partial^\star_{y_k}\bm{\varphi} \cdot\partial^\star_{y_l}\bm{\varphi} )
\\&+
2\mu_e\big[(\partial^\star_{y_i}\bm{\varphi} \cdot\partial^\star_{y_k}\bm{\varphi} )(\partial^\star_{y_j}\bm{\varphi} \cdot\partial^\star_{y_l}\bm{\varphi} )
\big]
\\&+2\mu_e\big[(\partial^\star_{y_i}\bm{\varphi} \cdot\partial^\star_{y_l}\bm{\varphi} )(\partial^\star_{y_j}\bm{\varphi} \cdot\partial^\star_{y_k}\bm{\varphi} )\big]
\end{align*}
are the contravariant components of $\mathbb{C}_e$ \cite[Page 162]{ciarlet2005introduction} where $\lambda_e$ and $\mu_e$ are the Lam\'e constants for the elastic shell satisfying
\begin{align}
\label{lameConstants}
3\lambda_e+2\mu_e>0, \qquad \mu_e>0.
\end{align} 
With this, the full nonlinear potential energy is given by the functional
\begin{equation}
\begin{aligned}
\label{koiterEnergyNonlinear}
K(\eta )=& \int_\Gamma
\frac{\epsilon_0}{2}  \mathbb{C}_e:\mathbb{G}_\eta \otimes \mathbb{G}_\eta   
  \dy_\bn 
 +
  \int_\Gamma 
\frac{\epsilon_0^3}{6}   \mathbb{C}_e:\mathbb{R}_\eta^\sharp  \otimes\mathbb{R}_\eta^\sharp   
 \dy_\bn  
 -
  \int_\Gamma
\left( 
\mathbf{g}\cdot\eta\bn
+g\eta
\right)
\dy_\bn
\end{aligned}
\end{equation}  
with the natural decomposition
$K(\eta )=K_m(\eta )+K_f(\eta )-\ell(\eta)$.  Here, $K_m $ is the  membrane part of the stored energy   due to stretching; $K_f$ is the flexural part of the stored energy due to bending; and $\ell$  consists of the force $\mathbf{g}$ applied to the shell from its interior and the surface force $g$ on the shell.

Having obtained  the full nonlinear potential  energy, we can now proceed to derive the equation of motion for the stochastically-constrained shell. For this, we first note that  if the mass density of $\Gamma $ is $\epsilon_0\varrho_s$ where $\varrho_s>0$ is a constant, the shell's momentum is $m=\epsilon_0\varrho_s\dot{\eta}$ so that its kinetic energy is given by
\begin{align*}
T(\dot{\eta})
=
\frac{1}{2}
 \int_\Gamma 
\epsilon_0 \varrho_s\vert \dot{\eta}\vert^2
  \dy_\bn
\end{align*}
recall \eqref{areaElem}.
To force the displacement field $\eta$ to follow the stochastic process above, we define the action to be the almost sure integral representation of the Lagrangian:
\begin{align}\label{mainAction}
&S(\eta,\dot{\eta})= \int_{t_0}^{t_1}\Big\{L(\eta,\dot{\eta})\dt
+
\int_\Gamma
m 
\Big(\dd \eta -\dot{\eta}  \dd t- \tfrac{1}{2}\sum_i\big(\bm{\sigma}_i \cdot\naby  \eta )\circ\dd W_t^i
\Big)\dy_\bn
\Big\} 
\end{align}
where
\begin{equation}
\begin{aligned}\label{ltk}
 L(\eta,\dot{\eta}) &:=T(\dot{\eta})- K(\eta )
\\&=
T(\dot{\eta})-K_m(\eta)-K_f(\eta)+\ell(\eta). 
\end{aligned}
\end{equation}
The stochastically constrained variational principle (Hamilton's principle) states that the actual motion $\eta(t,\by)$ of the shell makes the \textit{Action Integral} stationary over a fixed time interval $[t_1,t_2]$:
\begin{align}\label{zeroAction}
\delta S=0.
\end{align}
We can now seek a trajectory in the space of deformations that satisfies Hamilton's principle above subject to
\begin{align}\label{t0t1zero}
\delta\eta(t_0)=\delta\eta(t_1)=0.
\end{align} 
With this setup in place, we can now state our first main result.
\begin{theorem}\label{thm:main}
Stationarity \eqref{zeroAction} of  the action \eqref{mainAction}-\eqref{ltk} subject to \eqref{t0t1zero} result in the following stochastic Euler-Lagrange equation
\begin{equation}
\begin{aligned}   \nonumber
\epsilon_0\varrho_s  \dd \dot{\eta}  
&+
  \frac{\partial  }{\partial \eta}(K_m(\eta)+K_f (\eta))
 \dt
-
  (
 \mathbf{g}\cdot \bn
+g  )   \dt
\\&= 
\epsilon_0\varrho_s\sum_i \wp\big(\bm{\sigma}_i\cdot\naby \dot{\eta} 
+\tfrac{1}{2}\dot{\eta}\,\mathrm{div}_\by  \bm{\sigma}_i\big)\circ\dd W_t^i 
\end{aligned}
\end{equation}
 where $ \frac{\partial  f}{\partial \eta}$ is the 
  functional derivative of $f$ with respect to $\eta$.
\end{theorem} 
\begin{remark}
As is standard in elasticity, we refrain from  expressing $ \frac{\partial  f}{\partial \eta}$ in fully expanded coordinate form, since the resulting fourth-order nonlinear PDE  contains several dozen terms. This is not problematic, however, since  our interest later on  (see for example Section \ref{sec:visco}) is  the study of so-called \textit{weak solutions} to the shell equation, for which only the primitive $f$ of $ \frac{\partial  f}{\partial \eta}$ is required.
%
\end{remark} 
\begin{proof}[Proof of Theorem \ref{thm:main}]
By taking the first variation with respect to the displacement field, we obtain
\begin{align*}
\delta S(\eta,\dot{\eta})
&=
\frac{\dd}{\dd\tau} S(\eta+\tau \wp, \dot{\eta}+\tau \dot{\wp})\Big\vert_{\tau=0}, \qquad\qquad \wp(t_0)=\wp(t_1)=0 
\end{align*}
where $\tau>0$ is a scalar and $\wp(t)$ is an arbitrary test function. This result in several equivalent definitions. Indeed, for  $L(\eta,\dot{\eta}) $ the perturbation $\eta\mapsto\eta+\tau \wp$ results in
\begin{equation}
\begin{aligned}\label{lVariation}
\delta  \int_{t_0}^{t_1} L(\eta,\dot{\eta})\dt
&=  
\int_{t_0}^{t_1}\bigg(\frac{\partial L}{\partial \eta}  \wp\dt
+
\frac{\partial L}{\partial \dot{\eta}}   \dd \wp 
\bigg) 
\\
&=  
\int_{t_0}^{t_1} \bigg(\frac{\partial L}{\partial \eta}\wp\dt   
-
 \wp\dd 
\frac{\partial L}{\partial \dot{\eta}}   
\bigg)  
+
\frac{\partial L}{\partial \dot{\eta}}    \wp 
\bigg\vert_{t_0}^{t_1}
\\
&=  
\int_{t_0}^{t_1} \bigg(\frac{\partial L}{\partial \eta}\wp\dt   
-
\wp
 \dd 
\frac{\partial L}{\partial \dot{\eta}}   
\bigg).  
\end{aligned}
\end{equation}  
However, since $L$ decomposes linearly according to \eqref{ltk}, the (bi)-linear deterministic terms satisfy
\begin{align*}
&\int_{t_0}^{t_1}\bigg(\frac{\partial T}{\partial \eta}   
 \wp \dt-
  \wp\dd 
\frac{\partial T}{\partial \dot{\eta}}   
\bigg) 
=
-
\int_{t_0}^{t_1} \int_\Gamma \wp\,
\epsilon_0\varrho_s  \dd \dot{\eta} 
 \dy_\bn  
\\&
\int_{t_0}^{t_1}\bigg(\frac{\partial \ell}{\partial \eta} \wp\dt  
-
\wp
 \dd 
\frac{\partial \ell}{\partial \dot{\eta}}   
\bigg)   
=
\int_{t_0}^{t_1} \int_\Gamma
\left( 
\mathbf{g}\cdot \bn
+g 
\right) \wp  \dy_\bn \dt.
\end{align*} 
Consequently, it follows from \eqref{ltk} and \eqref{lVariation} that
\begin{equation}
\begin{aligned}\label{lVariation1}
\delta  \int_{t_0}^{t_1} L(\eta,\dot{\eta})\dt
=  
&-
\int_{t_0}^{t_1} \int_\Gamma \wp\,
\epsilon_0\varrho_s  \dd \dot{\eta} 
 \dy_\bn  
 -
 \int_{t_0}^{t_1} \frac{\partial  }{\partial \eta}(K_m+K_f )\wp \dt   
\\&+
 \int_{t_0}^{t_1} \int_\Gamma
 (
 \mathbf{g}\cdot \bn
+g  )  \wp  \dy_\bn \dt.
\end{aligned}
\end{equation}
We now recall that $m=\epsilon_0\varrho_s\dot{\eta}$. Thus, after integrating by parts and using $\wp(t_0)=\wp(t_1)=0 $, we obtain that
\begin{align*}
\delta  \int_{t_0}^{t_1} \int_\Gamma
m 
\,\dd \eta  \dy_\bn
&=   
\int_{t_0}^{t_1} \int_\Gamma
\epsilon_0\varrho_s \big(\dot{\eta} \dd\wp
+
\dot{\wp}\dd \eta  
\big)\dy_\bn 
\\
&=   -2
\int_{t_0}^{t_1} \int_\Gamma
\epsilon_0\varrho_s  \wp\dd\dot{\eta}   \dy_\bn 
\end{align*}
and
\begin{align*}
\delta  \int_{t_0}^{t_1} \int_\Gamma
m \,\dot{\eta}  \dd t \dy_\bn
&=  
\int_{t_0}^{t_1} \int_\Gamma\epsilon_0\varrho_s \bigg(\frac{\partial  \dot{\eta}^2}{\partial \eta} \wp\dt  
-
\wp \dd 
\frac{\partial  \dot{\eta}^2}{\partial \dot{\eta}}   
\bigg)   \dy_\bn
\\
&=  -2
\int_{t_0}^{t_1}\epsilon_0\varrho_s \wp \dd 
 \dot{\eta} \dy_\bn.
\end{align*}
Thus,
\begin{align}
\label{zeroDeter}
&\delta\int_{t_0}^{t_1}
\int_\Gamma
m\cdot
\big(\dd \eta -\dot{\eta}  \dd t
\big)\dy_\bn
  =0.
\end{align}
Finally, for the Stratonovich noise which satisfies the usual deterministic calculus, we have
\begin{equation}
\begin{aligned}\label{noneZerosStoch}
\delta  \int_{t_0}^{t_1} \int_\Gamma
m \sum_i\big(\bm{\sigma}_i \cdot\naby \eta)\circ\dd W_t^i
 \dy_\bn
&=  
  \int_{t_0}^{t_1} \int_\Gamma
\epsilon_0\varrho_s\sum_i  \dot{\wp}\big(\bm{\sigma}_i \cdot\naby \eta)\circ\dd W_t^i\dy_\bn
\\
&\quad+
  \int_{t_0}^{t_1} \int_\Gamma
\epsilon_0\varrho_s\sum_i \dot{\eta}\big(\bm{\sigma}_i \cdot  \naby\wp\big)\circ\dd W_t^i\dy_\bn
\\
&=  
- 2 \int_{t_0}^{t_1} \int_\Gamma
\epsilon_0\varrho_s\sum_i  \wp\big(\bm{\sigma}_i \cdot\naby \dot{\eta})\circ\dd W_t^i\dy_\bn
\\
&\quad-
  \int_{t_0}^{t_1} \int_\Gamma
\epsilon_0\varrho_s\sum_i \wp\,\dot{\eta}\,\mathrm{div}_\by\big(\bm{\sigma}_i \big)\circ\dd W_t^i\dy_\bn.
\end{aligned}
\end{equation}
If we now combine \eqref{lVariation1}, \eqref{zeroDeter} and \eqref{noneZerosStoch}, then it follows from \eqref{zeroAction} that
\begin{equation}
\begin{aligned}  \nonumber
\int_{t_0}^{t_1} \int_\Gamma \wp\,
\epsilon_0\varrho_s  \dd \dot{\eta} 
 \dy_\bn  
&+
 \int_{t_0}^{t_1}  
 \frac{\partial  }{\partial \eta}(K_m+K_f ) \wp 
 \dt
-
 \int_{t_0}^{t_1} \int_\Gamma
(
 \mathbf{g}\cdot \bn
+g  )  \wp   \dy_\bn \dt
\\&=
  \int_{t_0}^{t_1} \int_\Gamma
\epsilon_0\varrho_s\sum_i \wp\big(\bm{\sigma}_i\cdot\naby \dot{\eta} 
+\tfrac{1}{2}\dot{\eta}\,\mathrm{div}_\by  \bm{\sigma}_i\big)\circ\dd W_t^i\dy_\bn
\end{aligned}
\end{equation}
for any smooth test function $\wp$ satisfying $\wp(t_0)=\wp(t_1)=0$. Since this integral holds for any such test function, we immediately obtain the desired Euler--Lagrange equation. 
\end{proof}

\subsection{The linearized shell model}
To linearise the shell, we work in the same small-displacement regime already assumed above to ensure that the deformed basis $\mathcal{B}_\eta$ remains non-degenerate. 
We further use  \eqref{deformEta} along with   the identity 
\begin{equation}\label{eq:IdentityEta}
\mathbf{n} \cdot
 \partial_{y_j}\bm{\varphi}= \partial_{y_i}\bm{\varphi}\cdot \mathbf{n}=0
 \end{equation}
 to expand the covariant components $A_\eta^{ij}$, $i,j=1,2$ of the metric tensor $\mathbb{A}_\eta $  of  the deformed middle surface  $ \bm{\varphi}_{\eta(t)}( \Gamma )$ as
\begin{align*} 
A_\eta^{ij}=& \partial_{y_i} \bm{\varphi}  \cdot  \partial_{y_j} \bm{\varphi}  +  
\partial_{y_i} \bm{\varphi}  \cdot  \eta  \partial_{y_j}\mathbf{n} +   \eta \partial_{y_i}\mathbf{n} \cdot
  \partial_{y_j} \bm{\varphi}
\\&\quad  
  + 
   \partial_{y_i} \eta   \partial_{y_j} \eta 
   + 
    \partial_{y_i} \eta\,\mathbf{n}  \cdot  \eta  \partial_{y_j}\mathbf{n} 
  \\&\quad + 
   \eta \partial_{y_i}\mathbf{n} \cdot \partial_{y_j} \eta  \mathbf{n} 
  + 
    \eta  \partial_{y_i}\mathbf{n} \cdot  \eta  \partial_{y_j}\mathbf{n}.
\end{align*}   
Since only the first three terms in the expansion  above are linear (or affine) with respect to $\eta$, the covariant components of the linearised change of metric tensor $\mathbb{G}_\eta^{\mathrm{lin}}$ are given by
\begin{align*}
{G_\eta^{ij}}^{\mathrm{lin}} :=& \frac{1}{2}(A^{ij}_\eta-A^{ij})^{\mathrm{lin}}
\\
=& \frac{1}{2}\eta \big[\partial_{y_i} \bm{\varphi}  \cdot   \partial_{y_j}\mathbf{n} 
  +   \partial_{y_i}\mathbf{n} \cdot
  \partial_{y_j} \bm{\varphi} 
  \big], \qquad i,j=1,2.
\end{align*} 
Thus, recalling \eqref{nby} and using the properties of the scalar triple product together with \eqref{eq:IdentityEta}, we obtain 
\begin{align*}
\mathbb{G}_\eta^{\mathrm{lin}}=\frac{\eta }{\vert \partial_{y_1} \bm{\varphi}  \times \partial_{y_2} \bm{\varphi}  \vert}
\begin{pmatrix}
\partial_{y_1} \bm{\varphi}  \cdot(\partial_{y_1y_1}^2 \bm{\varphi} 
\times \partial_{y_2} \bm{\varphi} )
&
\partial_{y_1} \bm{\varphi}  \cdot(\partial_{y_1y_2}^2 \bm{\varphi} 
\times \partial_{y_2} \bm{\varphi} )\\
\partial_{y_1} \bm{\varphi}  \cdot(\partial_{y_1y_2}^2 \bm{\varphi} 
\times \partial_{y_2} \bm{\varphi} )
&
\partial_{y_1} \bm{\varphi} \cdot(\partial_{y_2y_2}^2 \bm{\varphi} 
\times \partial_{y_2} \bm{\varphi} ) 
\end{pmatrix}.
\end{align*} 
To derive  the  linearised modified change of curvature tensor $\mathbb{R}_\eta^{\sharp\mathrm{lin}}$, we first   extract the linear part of the covariant  components $B^{ij}_\eta$ of the curvature tensor. For this purpose, we note that  
\begin{align*}
B_\eta^{ij} 
:=& \mathbf{n}_{\eta }  \cdot
\partial_{y_iy_j}^2 \bm{\varphi}_{\eta } 
\\
=&(\partial_{y_1} \bm{\varphi}_{\eta} \times \partial_{y_2} \bm{\varphi}_{\eta}) \cdot 
\big(\partial_{y_iy_j}^2 \bm{\varphi}  
+ 
\partial_{y_iy_j}^2(\eta \bn ) 
\big)
\\
 =&
b_\eta^{ij}   
 +
 \mathrm{n.t.} 
\end{align*}
where $ \mathrm{n.t.} $ denotes the nonlinear terms (with respect to $\eta$), given by
\begin{align*}
\mathrm{n.t.}=& 
\Big[\big(\partial_{y_1}(\eta \mathbf{n}  )  \times  \partial_{y_2} \bm{\varphi}  \big) 
+
   \big(\partial_{y_1} \bm{\varphi}  \times  \partial_{y_2}(\eta  \mathbf{n}  )   \big)
 +
  \big(\partial_{y_1}(\eta \mathbf{n}   )  \times  \partial_{y_2}(\eta  \mathbf{n}  )  
  \big)
  \Big] \cdot 
\partial_{y_iy_j}^2(\eta \bn 
\big)
\\
&+
  \big(\partial_{y_1}(\eta \mathbf{n}  )  \times  \partial_{y_2}(\eta  \mathbf{n} )  
  \big) \cdot
 \partial_{y_iy_j}^2 \bm{\varphi},
\end{align*}
and
\begin{align*}
b_\eta^{ij} 
:=&
\vert \partial_{y_1} \bm{\varphi}  \times \partial_{y_2} \bm{\varphi} \vert 
\big(B^{ij}
+ 
\bn  \cdot \partial_{y_iy_j}^2(\eta \bn ) 
\big)
-
 \partial_{y_1}(\eta  \mathbf{n}  )  \cdot(    \partial_{y_iy_j}^2 \bm{\varphi} \times \partial_{y_2} \bm{\varphi}  )
-
  \partial_{y_2}(\eta  \mathbf{n}   )\cdot
(\partial_{y_1} \bm{\varphi}   \times  \partial_{y_iy_j}^2 \bm{\varphi}   )  .
\end{align*} 
Consequently, 
\begin{align*} 
{\vert \partial_{y_1} \bm{\varphi} \times \partial_{y_2} \bm{\varphi} \vert}^{-1} B_\eta^{ij}
=&   B^{ij}
+ 
\bn  \cdot \partial_{y_iy_j}^2(\eta \bn )   
-
  \partial_{y_1}(\eta  \mathbf{n}   )\cdot
  \mathbf{b}_{ij}^2
  -
  \partial_{y_2}(\eta  \mathbf{n}   )\cdot
  \mathbf{b}_{ij}^1
  +(
 \mathrm{n.t.})_2,  
\end{align*}
where
\begin{align*}
(
 \mathrm{n.t.})_2:=\frac{ 
 \mathrm{n.t.}}{\vert \partial_{y_1} \bm{\varphi} \times \partial_{y_2} \bm{\varphi} \vert}
\end{align*}
and
\begin{align*}
  \mathbf{b}_{ij}^1
  :=
  \frac{(\partial_{y_1} \bm{\varphi}   \times  \partial_{y_iy_j}^2 \bm{\varphi}   )  }{\vert \partial_{y_1} \bm{\varphi} \times \partial_{y_2} \bm{\varphi} \vert}
  ,\qquad
    \mathbf{b}_{ij}^2
  :=
  \frac{(   \partial_{y_iy_j}^2 \bm{\varphi} \times \partial_{y_2} \bm{\varphi}  )  }{\vert \partial_{y_1} \bm{\varphi} \times \partial_{y_2} \bm{\varphi} \vert}.
\end{align*} 
The covariant components of the linearised modified change of curvature tensor $\mathbb{R}_\eta^{\sharp \mathrm{lin}}$ is, therefore,
 \begin{equation}
\begin{aligned}\nonumber
 {R_\eta^{ij}}^{\sharp \mathrm{lin}} :=& \big({\vert \partial_{y_1} \bm{\varphi} \times \partial_{y_2} \bm{\varphi} \vert}^{-1} B_\eta^{ij}  -B^{ij}\big)^{\mathrm{lin}}
\\
=&
\bn  \cdot \partial_{y_iy_j}^2(\eta \bn )   
-
  \partial_{y_1}(\eta  \mathbf{n}   )\cdot
  \mathbf{b}_{ij}^2
  -
  \partial_{y_2}(\eta  \mathbf{n}   )\cdot
  \mathbf{b}_{ij}^1.
\end{aligned}  
\end{equation} 
Thus, the linearised change of curvature tensor $\mathbb{R}_\eta^{\sharp \mathrm{lin}}$ is 
\begin{align*}
 \mathbb{R}_\eta^{\sharp \mathrm{lin}}= 
\begin{pmatrix} 
\bn\cdot\partial_{y_1y_1}^2(\eta\bn)  - \partial_{y_1}(\eta\bn)  \cdot 
  \mathbf{b}_{11}^2
  - \partial_{y_2}(\eta\bn)  \cdot 
  \mathbf{b}_{11}^1
  &
   \bn\cdot\partial_{y_1y_2}^2(\eta\bn)  - \partial_{y_1}(\eta\bn)  \cdot
  \mathbf{b}_{12}^2
  - \partial_{y_2}(\eta\bn)  \cdot
  \mathbf{b}_{12}^1
\\
\bn\cdot\partial_{y_2y_1}^2(\eta\bn)  - \partial_{y_1}(\eta\bn) \cdot
  \mathbf{b}_{21}^2
  - \partial_{y_2}(\eta\bn) \cdot
  \mathbf{b}_{21}^1 
  &
  \bn\cdot\partial_{y_2y_2}^2(\eta\bn)  - \partial_{y_1}(\eta\bn)  \cdot
  \mathbf{b}_{22}^2
  - \partial_{y_2}(\eta\bn)  \cdot
  \mathbf{b}_{22}^1
\end{pmatrix}.
\end{align*}  
In analogy with the nonlinear energy \eqref{koiterEnergyNonlinear}, the linearised  energy is given by the functional
\begin{equation}
\begin{aligned}\nonumber
K^\mathrm{lin}(\eta)=& \int_\Gamma
\frac{\epsilon_0}{2} \mathbb{C}_e:\mathbb{G}_\eta^{\mathrm{lin}} \otimes \mathbb{G}_\eta^{\mathrm{lin}}   \dy_\bn 
 +
  \int_\Gamma 
\frac{\epsilon_0^3}{6}   \mathbb{C}_e:\mathbb{R}_\eta^{\sharp \mathrm{lin}}  \otimes\mathbb{R}_\eta^{\sharp\mathrm{lin}}   \dy_\bn  
 -
  \int_\Gamma
\left( 
\mathbf{g}\cdot\eta\bn
+ g\eta
\right)  \dy_\bn 
\end{aligned}
\end{equation} 
with the natural decomposition
$K^\mathrm{lin}(\eta )=K^\mathrm{lin}_m(\eta )+K^\mathrm{lin}_f(\eta )-\ell(\eta)$.  
As a result, if the mass density of $\Gamma$ is $\epsilon_0\varrho_s$ where $\varrho_s>0$ is a constant, the weak formulation of the linearised elastic shell is given by  
\begin{equation}
\begin{aligned}  \label{koiterEnergyLinear1}
\int_{t_0}^{t_1} \int_\Gamma \wp\,
\epsilon_0\varrho_s  \dd \dot{\eta} 
 \dy_\bn  
&
+
 \int_{t_0}^{t_1}  
 \frac{\partial  }{\partial \eta}(K^{\mathrm{lin}}_m+K^{\mathrm{lin}}_f ) \wp 
 \dt
-
 \int_{t_0}^{t_1} \int_\Gamma
 (
 \mathbf{g}\cdot \bn
+g  )  \wp  \dy_\bn \dt 
\\&=
  \int_{t_0}^{t_1} \int_\Gamma
\epsilon_0\varrho_s\sum_i \wp\big(\bm{\sigma}_i\cdot\naby \dot{\eta} 
+\tfrac{1}{2}\dot{\eta}\,\mathrm{div}_\by  \bm{\sigma}_i\big)\circ\dd W_t^i\dy_\bn
\end{aligned}
\end{equation}
for all test functions satisfying $\wp(t_0)=\wp(t_1)=0 $. The following result is, therefore,  a direct  corollary of Theorem \ref{thm:main}:  
\begin{corollary}\label{cor:main}
Stationarity \eqref{zeroAction} of  the action \eqref{mainAction} subject to \eqref{t0t1zero}
where
\begin{align*}
 L(\eta,\dot{\eta}) &:=T(\dot{\eta})- K(\eta )
\\&=
T(\dot{\eta})-K^{\mathrm{lin}}_m(\eta)-K^{\mathrm{lin}}_f(\eta)+\ell(\eta)
\end{align*} results in the following stochastic Euler-Lagrange equation
\begin{equation}
\begin{aligned}   \nonumber
\epsilon_0\varrho_s  \dd \dot{\eta}  
&+
  \frac{\partial  }{\partial \eta}(K^{\mathrm{lin}}_m(\eta)+K^{\mathrm{lin}}_f (\eta))
 \dt
-
  (
 \mathbf{g}\cdot \bn
+g  )   \dt
\\&= 
\epsilon_0\varrho_s\sum_i \wp\big(\bm{\sigma}_i\cdot\naby \dot{\eta} 
+\tfrac{1}{2}\dot{\eta}\,\mathrm{div}_\by  \bm{\sigma}_i\big)\circ\dd W_t^i 
\end{aligned}
\end{equation}
 where $ \frac{\partial  f}{\partial \eta}$ is the 
  functional derivative of $f$ with respect to $\eta$.
\end{corollary}

\section{Prototype for the simplified linear model.} 
\label{sec:proto}
We recall that the Euler--Lagrange equations for both the nonlinear Koiter shell (Theorem \ref{thm:main}) and the linear Koiter shell (Corollary \ref{cor:main}) are expressed  in terms of  abstract functional derivatives of their respective  membrane and flexural energies. In order to obtain an exact expression for these derivatives, we now consider  a simplified linear model, in which  some lower-order terms are neglected and  the coefficients in the potential energy are assumed to be constant. In particular, we assume that the parameter
\begin{align*}
  \nu_e:=\frac{\epsilon_0}{4}   C_e^{ijkl} \big[\partial_{y_i} \bm{\varphi}  \cdot   \partial_{y_j}\mathbf{n} 
  +   \partial_{y_i}\mathbf{n} \cdot
  \partial_{y_j} \bm{\varphi} 
  \big], 
\end{align*}
is a uniform constant. Then, since 
\begin{align*}
{G_\eta^{ij}}^{\mathrm{lin}} 
=& \frac{1}{2}\eta \big[\partial_{y_i} \bm{\varphi}  \cdot   \partial_{y_j}\mathbf{n} 
  +   \partial_{y_i}\mathbf{n} \cdot
  \partial_{y_j} \bm{\varphi} 
  \big], 
\end{align*} 
it follows that 
\begin{equation}
\begin{aligned}
\label{koiterEnergyLinearG}
K_{m,s}^\mathrm{lin}(\eta)
:=
 \int_\Gamma
\frac{\epsilon_0}{2}  \mathbb{C}_e:\mathbb{G}_\eta^{\mathrm{lin}} \otimes \mathbb{G}_\eta^{\mathrm{lin}}    \dy_\bn 
=
\frac{\nu_e}{2}
 \int_\Gamma     \eta^2 \dy_\bn.
\end{aligned}
\end{equation}  
Importantly, this simplification is justified by the coercivity of the membrane energy. Indeed, since $\mathbb{G}_\eta^{\mathrm{lin}}$ is symmetric, and the Lam\'e coefficients $\lambda_e, \mu_e$  satisfy \eqref{lameConstants}, it  follows from \cite[Theorem 4.4-1]{ciarlet2005introduction} that there exist a constant $c=c(\lambda_e,\mu_e,\epsilon_0,\bm{\varphi},\Gamma)>0$ such that
\begin{align*}
& \int_\Gamma
\frac{\epsilon_0}{2} \mathbb{C}_e:\mathbb{G}_\eta^{\mathrm{lin}} \otimes \mathbb{G}_\eta^{\mathrm{lin}}   \dy_\bn \geq \frac{c}{2} \int_\Gamma
\eta^2   \dy_\bn.
\end{align*}
We may, therefore,  choose $\nu_e=c>0$ consistently  with this coercivity bound, so that, in view of  \eqref{koiterEnergyLinearG}, we derive that 
\begin{align} \label{koiterEnergyLinear1a}
 &\int_{t_0}^{t_1}\frac{\partial K^{\mathrm{lin}}_{m,s}}{\partial \eta} \wp    \dt 
 =  
\int_{t_0}^{t_1} \int_\Gamma
 \nu_e  \eta  \, \wp \dy_\bn \dt.
\end{align} 
We now derive a simplified expression for the flexural energy associated with  the linearised change of curvature tensor. We start by neglecting lower-order terms and, in analogy with the membrane case,  assume that the coefficients of the higher-order terms are constant. This leads to 
\begin{equation}
\begin{aligned}\nonumber
K_{f,s}^\mathrm{lin}(\eta)
:=
  \int_\Gamma 
\frac{\epsilon_0^3}{6}&   \mathbb{C}_e:\mathbb{R}_\eta^{\sharp\mathrm{lin}}  \otimes\mathbb{R}_\eta^{\sharp\mathrm{lin}}    \dy_\bn  
=
\frac{1}{2}
  \int_\Gamma 
 (   \alpha|\nabla^2_{\by}\eta|^2+\beta |\nabla_{\by}\eta|^2) \dy_\bn  
\end{aligned}
\end{equation}
where the terms $\alpha>0$ and $\beta>0$ are constants. Then, considering the perturbation $\eta\mapsto\eta+\tau \wp$, we obtain after integrating by parts, that
\begin{align*}
\delta \int_{t_0}^{t_1} \int_\Gamma\frac{\alpha}{2}\vert\naby^2\eta\vert^2\dy_\bn\dt
&=  
\frac{\dd}{\dd\tau}
\int_{t_0}^{t_1}\int_\Gamma\frac{\alpha}{2}
\big(\vert\naby^2\eta \vert^2
+2\tau
\naby^2 \eta:\naby^2\wp 
+
\tau^2
\vert\naby^2 \wp\vert^2
\big)
\dy_\bn
\dt\Big\vert_{\tau=0}
\\
&=  \alpha
\int_{t_0}^{t_1}\int_\Gamma \naby^2 \eta:\naby^2\wp 
\dy_\bn
\dt
\\
&=  \alpha
\int_{t_0}^{t_1}\int_\Gamma \wp\,\Delta_\by^2 \eta  
\dy_\bn
\dt
\end{align*}
and similarly,
\begin{align*}
\delta \int_{t_0}^{t_1} \int_\Gamma\frac{\beta}{2}\vert\naby\eta\vert^2\dy_\bn\dt
&=    
\frac{\dd}{\dd\tau}
\int_{t_0}^{t_1}\int_\Gamma\frac{\beta}{2}
\big(\vert\naby\eta \vert^2
+2\tau
\naby \eta\cdot\naby\wp 
+
\tau^2
\vert\naby \wp\vert^2
\big)
\dy_\bn
\dt\Big\vert_{\tau=0}
\\
&=  \beta
\int_{t_0}^{t_1}\int_\Gamma \naby \eta\cdot\naby\wp 
\dy_\bn
\dt
\\
&= - \beta
\int_{t_0}^{t_1}\int_\Gamma \wp\, \Delta_\by \eta  
\dy_\bn
\dt.
\end{align*}
Therefore,
\begin{align} \label{koiterEnergyLinear1b}
 &\int_{t_0}^{t_1}\frac{\partial K^{\mathrm{lin}}_{f,s}}{\partial \eta}    \wp\dt
 =  
\int_{t_0}^{t_1} \int_\Gamma
 (    \alpha \Delta_{\by}^2\eta- \beta \Delta_{\by}\eta )   \wp \dy_\bn \dt.
\end{align}
Substituting \eqref{koiterEnergyLinear1a} and 
\eqref{koiterEnergyLinear1b} into \eqref{koiterEnergyLinear1}, we   arrive at 
\begin{equation}
\begin{aligned}  \nonumber
\int_{t_0}^{t_1} \int_\Gamma \wp\,
\epsilon_0\varrho_s  \dd \dot{\eta} 
 \dy_\bn  
&+
 \int_{t_0}^{t_1} \int_\Gamma
 \wp ( \nu_e\eta  + \alpha \Delta_{\by}^2\eta- \beta \Delta_{\by}\eta
 -
 \mathbf{g}\cdot \bn
-g  )   \dy_\bn \dt
\\&=
  \int_{t_0}^{t_1} \int_\Gamma
\epsilon_0\varrho_s\sum_i \wp\big(\bm{\sigma}_i\cdot\naby \dot{\eta} 
+\tfrac{1}{2}\dot{\eta}\,\mathrm{div}_\by  \bm{\sigma}_i\big)\circ\dd W_t^i\dy_\bn
\end{aligned}
\end{equation}
for all test functions $\wp$  satisfying $\wp(t_0)=\wp(t_1)=0 $. Since this identity holds for all such test functions, we deduce the Euler--Lagrange equation 
\begin{equation}
\begin{aligned}   \label{e-l}
\epsilon_0\varrho_s  \dd \dot{\eta}  
&+ 
( \nu_e\eta +   \alpha \Delta_{\by}^2\eta- \beta \Delta_{\by}\eta  
 -
 \mathbf{g}\cdot \bn
-g  )   \dt
\\&= 
\epsilon_0\varrho_s\sum_i   \big(\bm{\sigma}_i\cdot\naby \dot{\eta} 
+\tfrac{1}{2}\dot{\eta}\,\mathrm{div}_\by  \bm{\sigma}_i\big)\circ\dd W_t^i,\qquad\qquad\text{ on}\qquad I\times\Gamma
\end{aligned}
\end{equation} 
complemented by suitable initial conditions
for $\eta$ and $\dot{\eta}$.  Importantly,  a purely elastic, divergence-free variant of \eqref{e-l} has already seen application in fluid-structure interaction \cite{breit2024martingale}.
 
\subsection{The coefficients}
We conclude this section with a brief discussion of the various terms in \eqref{e-l} as they are intrinsic to the generalised nonlinear and linear Koiter shell models explored earlier. The zero-order term $\nu_e\eta$ represents stiffness due to shell curvature and serves as a damping term. Its coefficient $\nu_e$ changes sign depending on whether the shell is pulled tight or compressed. When $\nu_e>0$, the shell is in a stable regime, where tension acts as an additional restoring force. In contrast, the unstable regime corresponds to $\nu_e<0$, where compression acts as a softening force that can  lead to buckling. In the neutral regime
$\nu_e=0$, restoring forces arise solely from bending, as  encoded by the fourth-order bending elasticity. This neutral regime corresponds to the behaviour observed in  plates.  As already mentioned, $\alpha\Delta_\by^2\eta$ is the fourth-order bending elasticity term. It constitutes, probably, the most important feature of thin-shell models, as it captures the rigidity or stiffness of bending. Its coefficient $\alpha>0$ is related to the thickness of the shell and the larger it is, the thicker the shell is.
The membrane elasticity operator
$-\beta\Delta_\by\eta$ can be viewed as the second-order analogue of the fourth-order bending elasticity term. While  the latter accounts for bending stress, the former carries information on membrane stress. When  $\beta>0$, the shell is in a stable regime, where  stiffness resists deformation. However, when  $\beta<0$, the shell enters an unstable regime, where  negative stiffness  amplifies deformation and may lead to buckling. In the neutral regime $\beta=0$, there is no membrane load, and the shell exhibit purely bending behaviour.

 \section{Viscoelastic effect from transport noise}\label{sec:visco}
 Recent works \cite{flandoli2021scaling, flandoli2020convergence, flandoli2021high,  galeati2020convergence} have demonstrated the regularising effects of transport noise in fluids. Our goal in this section is to demonstrate that a similar result can be applied to the elastic materials under study. Indeed, we show that for a suitably chosen family of transport noise, certain solutions of the constraint models derived in the previous section,  that  are parametrised by this family of noise, regularises the shell equation in a certain asymptotic regime. For the purpose of clarity, we will demonstrate this result only for the prototype \eqref{e-l} but this result also applies to the earlier linearised Koiter shell model in  Corollary \ref{cor:main}. The nonlinear Koiter shell models in Theorem \ref{thm:main}, however, will require additional work and is currently not covered by the subsequent analysis.
 
To begin with, we wish to choose $\bm{\sigma}_i$ as divergence-free vector fields where the enumeration $i$ corresponds to the
increasing rearrangement $|\bk|^2$ of the modes or wavevectors $\bk$ in the punctured lattice $\mathbb{Z}^2_{\bm{0}}:=\mathbb{Z}^2\setminus\{\bm{0}\}$. Since the shell unknown  is a real-valued function, it is expected that any driving force for its evolution is also real-valued.  Somewhat ironically,  however,  to obtain our desired dissipation effect in the shell, we first need to construct ``artificial complex-valued" Brownian motions. Their role is purely analytical, as it allows us to write the transport noise in a complex Fourier basis while ensuring, through the pairing of opposite modes, that the resulting Fourier series has real coefficients and hence defines a real-valued forcing. This would become clearer as we proceed with the construction. 
First, we consider a disjoint partition of the punctured lattice $\mathbb{Z}^2_{\bm{0}}=\mathbb{Z}^2_{+}\cup \mathbb{Z}^2_{-}$ where
\begin{align*}
\mathbb{Z}^2_{+}:=\{\bk=(k_1,k_2)\in\mathbb{Z}^2_{\bm{0}}
\,:\, (k_1>0) \text{ or }(k_1=0,k_2>0)\}
\end{align*}
represents the  ``positive" half of the integer lattice $\mathbb{Z}$ based on lexicographical ordering and $\mathbb{Z}^2_{-}=-\mathbb{Z}^2_{+}$ represents the ``negative" half. We now consider the family $(W_t^\bk)_{t\geq0}$ of complexified Brownian motions defined by
\begin{align}\label{BMimaginary}
W_t^\bk
= \left\{
  \begin{array}{lr}
 B_t^\bk+i \,B_t^{-\bk} & \text{ if }\bk\in \mathbb{Z}^2_{+}\\
B_t^{-\bk}-i \,B_t^{\bk} & \text{ if }\bk\in \mathbb{Z}^2_{-}
  \end{array}
\right.
\end{align}
where $(B_t^\bk)_{t\geq0}$ is the usual family of real-valued, independent, identically distributed Brownian motions. Note that since the quadratic covariation of two complex martingales $M=X+iY$ and $N=U+iV$ is given by
\begin{align*}
[M,N]=[X,U]-[Y,V]+i \big([X,V] +[Y,U]\big),
\end{align*}
it follows that
\begin{align}\label{qv}
\left[W_t^\bk,W_t^\mathbf{\ell}\right]
=  
2t\,\delta_{\bk,-\mathbf{\ell}}
\end{align}
whereas the Hermitian Covariation satisfies
\begin{align*}
\left[W_t^\bk,\overline{W_t^\mathbf{\ell}}\right]
=  
\left[W_t^\bk, W_t^{-\mathbf{\ell}}\right]
=  
2t\,\delta_{\bk,\mathbf{\ell}}.
\end{align*} 
We also observe that the construction of $(W_t^\bk)_{t\geq0}$ leads to the relation $W_t^{-\bk}=\overline{W_t^\bk }$ that is required for a strictly real-valued coefficient of a Fourier series. These  Brownian motions provide the desired stochastic amplitudes for the Fourier transport modes. What remains is to incorporate incompressibility in the driving force. This would consist of  noise coefficients that maps non-zero integer frequency vectors $\bk\in\mathbb{Z}^2_{\bm{0}}$ to specific vectors in $\mathbb{R}^2$.  
More precisely,  we consider the orthonormal basis
$(e^{i\bk\cdot\by})_{\bk\in\mathbb{Z}^2}$ of the space $L^2(\Gamma)$ formed by trigonometric
functions and for any $N\in\mathbb{N}$, set
\begin{align*}
H^N = \mathrm{span}\left\{e
^{i\bk\cdot\by} \text{ with } |\bk| \leq N\right\}
\end{align*}
with the associated $L^2$-orthogonal projection $P^N:L^2(\Gamma)\rightarrow H^N$. Now, for $\bk=(k_1,k_2)^\top\in\mathbb{Z}^2_{\bm{0}}$ with $\bk^\perp=(-k_2,k_1)^\top$, we consider the ansatz 
\begin{equation}\label{ansatz}
\left\{\begin{aligned}
\bm{\sigma}^N_\bk(\by) & =  \frac{\sqrt{2\gamma}}{\epsilon_0\varrho_s}\frac{i\,\bk^\perp}{|\bk|^2}\bm{1}_{\{N\leq|\bk|\leq 2N\}}
e^{i\,\bk\cdot\by}\Big(\sum_{N\leq|\bk|\leq 2N}\frac{1}{|\bk|^2}\Big)^{-1/2}, \qquad \gamma>0
\\[0.4em]
\bm{\sigma}^N_{\bm{0}} (\by) & = {\bm{0}} ,
\end{aligned} \right.
\end{equation}
which is similar to the coefficient considered in \cite{flandoli2020convergence} but multiplied by the imaginary unit $i$ and also rescaled to suit our setting. Note that the choice of $\bm{\sigma}^N_\bk$  is by no means  unique. Indeed, several alternative constructions of transport noise coefficients appear in the literature; see, for instance,  \cite{flandoli2021scaling, flandoli2020convergence, flandoli2021high,  galeati2020convergence}. Nevertheless, $\bm{\sigma}^N_\bk$ as defined above is suitable for our purpose as it is incompressible (in Fourier or frequency space) since $\bk^\perp\cdot\bk=0$ and it is real-valued since $\bm{\sigma}^N_{-\bk}=\overline{\bm{\sigma}^N_\bk}$. Additionally, one can verify that it is Lipschitz continuous. Furthermore, combining the relation $\bm{\sigma}^N_{-\bk}=\overline{\bm{\sigma}^N_\bk}$ with the fact that $W_t^{-\bk}=\overline{W_t^\bk}$, it follows that  for any fixed $t\in I$,
\begin{align*}  
\sum_{\bk\in\mathbb{Z}^2_{\bm{0}}}\bm{\sigma}^N_\bk( \by)   W_t^\bk 
          =
          \frac{\sqrt{2\gamma}}{\epsilon_0\varrho_s}
          \Big(\sum_{N\leq|\bk|\leq 2N}\frac{1}{|\bk|^2}\Big)^{-1/2}
\sum_{N\leq|\bk|\leq 2N} \frac{i\,\bk^\perp}{|\bk|^2} e^{i\bk\cdot\by}   W_t^\bk 
\end{align*}
is a  (truncated) Fourier series of a well-defined, real-valued, divergence-free 
random field. 
By setting $\eta_0^N=P^N \eta_0$, $g^N=P^N g$ and $ \mathbf{g}^N=P^N  \mathbf{g}$,  our goal now is to search for coefficients $\varsigma_\bk^N:\Sigma\times I\rightarrow\mathbb{R}$ such that
\begin{align*}
\eta^N(t,\by)=\sum_{|\bk|\leq N} \int_0^t\varsigma_\bk^N(s)e^{i\bk\cdot\by}\ds+\eta_0^N
\end{align*}
solves the following system 
 \begin{equation}
\left\{ \begin{aligned}\label{sde1}
&\epsilon_0\varrho_s  \dd \dot{\eta}^N  
+ 
( \nu_e\eta^N +   \alpha \Delta_{\by}^2\eta^N- \beta \Delta_{\by}\eta^N  
 -
 \mathbf{g}^N\cdot \bn
-g^N  )   \dt
=
 \epsilon_0\varrho_s 
\sum_{\bk\in\mathbb{Z}^2_{\bm{0}}} \bm{\sigma}^N_\bk  \cdot\naby\dot{\eta}^N  \circ  \dd W_t^\bk 
,
 \\
&\eta_\star^N=\partial_t\eta^N(0), \qquad
 \eta_0^N=P^N \eta_0.
\end{aligned} \right.
 \end{equation}
with  the initial condition $\varsigma_\bk^N(0)$ chosen such that\footnote{For example, $\varsigma_\bk^N(0)\equiv\widehat{\eta}_{\star,\bk}$ where $\widehat{\eta}_{\star,\bk}=\int_\Gamma e^{-i\bk\cdot\by}\eta_\star(\by)\dy$}
\begin{align*}
 \eta_\star^N=\partial_t\eta^N(0) \rightarrow \eta_\star\qquad\text{ in }\qquad L^2(\Gamma).
\end{align*} 
By Lemma \ref{lemA1}, the stochastic transport term can be rewritten in It\^o form as 
\begin{align*}
\epsilon_0\varrho_s\ 
\sum_{\bk\in\mathbb{Z}^2_{\bm{0}}} \bm{\sigma}^N_\bk  \cdot\naby\dot{\eta}^N  \circ  \dd W_t^\bk 
=&  \epsilon_0\varrho_s
\sum_{\bk\in\mathbb{Z}^2_{\bm{0}}} \bm{\sigma}^N_\bk \cdot\naby\dot{\eta}^N    \dd W_t^\bk 
    +
\gamma \Dely \dot{\eta}^N
 \dt  .
\end{align*}  
Thus,  the finite-dimensional  SDE  \eqref{sde1} is equivalent, via the Stratonovich-to-It\^o map,  to
 \begin{equation}
\left\{ \begin{aligned}\label{sde2}
&\epsilon_0\varrho_s  \dd \dot{\eta}^N  
+ 
( \nu_e\eta^N +   \alpha \Delta_{\by}^2\eta^N- \beta \Delta_{\by}\eta^N  
 -
 \mathbf{g}^N\cdot \bn
-g^N  -\gamma\Dely\dot{\eta}^N)   \dt
\\&\qquad\qquad= 
\epsilon_0\varrho_s
\sum_{\bk\in\mathbb{Z}^2_{\bm{0}}}     \bm{\sigma}^N_\bk\cdot\naby \dot{\eta}^N 
 \dd W_t^\bk,
 \\
&\eta_\star^N=\partial_t\eta^N(0), \qquad
 \eta_0^N=P^N \eta_0.
\end{aligned} \right.
\end{equation} 
The precise notion of a solution of \eqref{sde2} is now given as follows:
 \begin{definition}[Weak pathwise solution]
\label{def:weakSolution}
Let $(\eta_0, \eta_\star,  g, \mathbf{g},(\bm{\sigma}^N_\bk)_{\bk\in\mathbb{Z}_{\bm{0}}})$ be a dataset such that
\begin{equation}
\begin{aligned}
\label{datasetStrongSol}
&
\eta_0 \in W^{2,2}(\Gamma ) \text{ with } \Vert \eta_0 \Vert_{L^\infty( \Gamma )} < L, \quad
\eta_\star \in L^{2}(\Gamma ),  
\\&
g, \mathbf{g} \in L^{4}(I;L^{2}(\Gamma )), \quad
\Vert\bm{\sigma}^N_\bk\Vert_{W^{1,\infty}(\Gamma)}\lesssim 1.
\end{aligned}
\end{equation} 
Also, let $(\Sigma,\mathcal{F},(\mathcal{F}_t)_{t\geq0},\mathbb{P})$  be a stochastic basis and let $(W^\bk_t)_{t\geq0}$ be a Brownian motion adapted to the complete right-continuous filtration $(\mathcal{F}_t)_{t\geq0}$.
We call $\eta^N$
a \textit{weak pathwise solution}  of \eqref{sde2} with Fourier truncated data $(\eta^N_0, \eta^N_\star,  g^N, \mathbf{g}^N,(\bm{\sigma}^N_\bk)_{\bk\in\mathbb{Z}_{\bm{0}}})$ provided that the following holds:
\begin{itemize}
\item[(a)]  $ \eta^N $ is $(\mathcal{F}_t)$-adapted with
\begin{align*}
&\eta^N \in L^{\infty} (I;W^{2,2}(\Gamma ))
,
\qquad
\dot{\eta}^N \in  C_w (\overline{I};L^{2}(\Gamma ))  \qquad \text{a.s.};
\end{align*} 
\item[(b)] the equation
\begin{align*}  
\epsilon_0\varrho_s\int_0^t\int_{\Gamma}\dot{\eta}^N\dot{\zeta}\dy\ds
=
&\int_0^t\int_{\Gamma}( \nu_e\eta^N\zeta   
+\alpha\Dely \eta^N\Dely \zeta + \beta \naby\eta^N\cdot\naby\zeta )\dy\ds
\\
&-\int_0^t\int_{\Gamma}( g^N\zeta+\mathbf{g}^N\cdot\bn\zeta  +\gamma\dot{\eta}^N\Dely\zeta)\dy\ds
 \\
&+
 \epsilon_0\varrho_s
\int_0^t\int_{\Gamma} \sum_{\bk\in\mathbb{Z}^2_{\bm{0}}} \dot{\eta}^N\bm{\sigma}^N_\bk\cdot\naby \zeta 
\dy\dd W_s^\bk
\end{align*} 
holds $\mathbb{P}$-a.s. for a.e. $t\in \overline{I}$ and for all $\zeta\in C^\infty_c(I \times\Gamma)$.
\end{itemize}
\end{definition}
Note that this solution is  weak in the PDE sense (equation holds weakly in the sense of distributions) but strong in the stochastic sense (the solution is defined on a given stochastic basis with a given family of Brownian motions). Compared to the weak-weak solution usually explored in the literally, the linear structure of the shell regularises the solution in the stochastic sense.
Now, with this precise notion of a solution given and the prior  preparatory framework, we can now state the main result of this section.
 \begin{theorem}\label{thm:main1}
Fix $\epsilon_0,\varrho_s,\alpha,\gamma>0$ and $\nu_e,\beta\geq0$   and let $ (\Sigma,\mathcal{F},(\mathcal{F}_t)_{t\geq0},\mathbb{P})$ be a stochastic basis. Assume that $(\eta_0, \eta_\star,  g, \mathbf{g} )$ are such that
\begin{align*}
&
\eta_0 \in W^{2,2}(\Gamma ) \text{ with } \Vert \eta_0  \Vert_{L^\infty( \Gamma )} < L, \quad
\eta_\star  \in L^{2}(\Gamma ),  
\\&
g, \mathbf{g} \in L^{4}(I;L^{2}(\Gamma )).
\end{align*}
$\mathbb{P}$-a.s. Then we can find incompressible vector fields $(\bm{\sigma}^N_\bk)_{\bk\in\mathbb{Z}^2_{\bm{0}}}$ satisfying $\Vert\bm{\sigma}_\bk^N\Vert_{W^{1,\infty}(\Gamma)}\lesssim 1$ and a corresponding family $(W_t^\bk)_{t\geq0}$ of Brownian motions adapted to $(\mathcal{F}_t)_{t\geq0}$ such that a unique weak pathwise solution $\eta^N$ of \eqref{sde2}
with  Fourier truncated  data $\left(\eta_0^N, \eta_\star^N,   g^N, \mathbf{g}^N,(\bm{\sigma}_i^N)_{i=1}^N \right)$  exists\footnote{Note that since this latter dataset are truncated Fourier series of the original dataset and their individual regularities are at worst square-integrable, the latter automatically converges pointwise almost everywhere to their corresponding originals.}.  Furthermore, 
up to subsequence (not relabelled)
\begin{align*}
\eta^N \overset{*}{\rightharpoonup} \eta \text{ in }& L^\infty(I;W^{2,2}(\Gamma)),
\\
\dot{\eta}^N \overset{*}{\rightharpoonup}  \dot{\eta}
\text{ in }& L^\infty(I;L^{2}(\Gamma)),
\end{align*}
$\mathbb{P}$-a.s. where $\eta$ is the unique global weak solution  of
\begin{align*}
\epsilon_0\varrho_s  \, \ddot{\eta}  
&+ 
 \nu_e\eta +   \alpha \Delta_{\by}^2\eta- \beta \Delta_{\by}\eta
-
 \mathbf{g}\cdot \bn -g -
 \gamma
 \Dely\dot{\eta}  
 =0
\end{align*}
with data $(\eta_0, \eta_\star, g, \mathbf{g} )$. Here, $\ddot{\eta} :=\partial_t^2\eta$, $\dot{\eta} :=\partial_t\eta$ and by `weak solution' for the limit system, we mean that
\begin{align*}  
\epsilon_0\varrho_s\int_0^t\int_{\Gamma}\dot{\eta}\dot{\zeta}\dy\ds
=
&\int_0^t\int_{\Gamma}( \nu_e\eta\zeta   
+\alpha\Dely \eta\Dely \zeta + \beta \naby\eta\cdot\naby\zeta 
-g\zeta-\mathbf{g}\cdot\bn\zeta  -\gamma\dot{\eta}\Dely\zeta)\dy\ds 
\end{align*} 
holds for a.e. $t\in \overline{I}$ and for all $\zeta\in C^\infty_c(I \times\Gamma)$.
\end{theorem} 
\begin{remark}
Whereas Theorem \ref{thm:main1} is performed for deterministic dataset $(\eta_0, \eta_\star,  g, \mathbf{g} )$ satisfying \eqref{datasetStrongSol}, a careful analysis of the subsequent proof shows that it can be extended to random variables $(\eta_0, \eta_\star,  g, \mathbf{g} )$ provided that they satisfy
\begin{align*}
&
\eta_0 \in L^4(\Sigma;W^{2,2}(\Gamma )) \text{ with } \Vert \eta_0  \Vert_{L^\infty( \Gamma )} < L \text{ a.s}, \quad
\eta_\star  \in L^4(\Sigma;L^{2}(\Gamma )),  
\\&
g, \mathbf{g} \in L^{4}(\Sigma\times I;L^{2}(\Gamma ))
\end{align*}
and their corresponding truncated dataset converges almost surely to them.
\end{remark}
\begin{proof}[Proof of Theorem \ref{thm:main1}] 
Since  \eqref{sde2} is a linear finite-dimensional SDE with global Lipschitz coefficients, standard finite-dimensional SDE theory (see, e.g.,   \cite[Chapter 5, Theorem 2.9]{karatzas2014brownian} and Yamada--Watanabe Theorem) guarantees the existence of a unique strong stochastic solution on the prescribed stochastic basis. Consequently to establish the first part of Theorem \ref{thm:main1}, it only remains to verify that this solution satisfies the regularity requirements in  item (a) of Definition \ref{def:weakSolution}.


%
%
For this purpose, we apply It\^o's formula to the mapping $t\mapsto \tfrac{1}{2}\Vert\dot{\eta}^N\Vert_{L^2(\Gamma)}^2$. This yields
\begin{align*}
\frac{1}{2}&\Big( \epsilon_0\varrho_s\Vert\dot{\eta}^N(t)\Vert_{L^2(\Gamma)}^2+\nu_e \Vert \eta^N(t)\Vert_{L^2(\Gamma)}^2+\alpha \Vert \Dely\eta^N(t)\Vert_{L^2(\Gamma)}^2+\beta \Vert \naby\eta^N(t)\Vert_{L^2(\Gamma)}^2\Big)+\gamma \int_0^t\Vert \naby\dot{\eta}^N \Vert_{L^2(\Gamma)}^2\ds
\\&=
\frac{1}{2}\Big( \epsilon_0\varrho_s\Vert \eta_\star^N\Vert_{L^2(\Gamma)}^2+\nu_e \Vert \eta_0^N\Vert_{L^2(\Gamma)}^2+\alpha \Vert \Dely\eta_0^N\Vert_{L^2(\Gamma)}^2+\beta \Vert \naby\eta_0^N\Vert_{L^2(\Gamma)}^2\Big)
+
\int_0^t\int_\Gamma(g^N+\mathbf{g}^N\cdot\bn)\dot{\eta}^N\dy\ds
 \\&+
  \epsilon_0^2\varrho_s^2  \int_0^t\sum_{\bk\in\mathbb{Z}^2_{\bm{0}}}  \Vert\bm{\sigma}^N_\bk\cdot \naby\dot{\eta}^N \Vert_{L^2(\Gamma)}^2\ds 
+
 \epsilon_0\varrho_s \int_0^t\sum_{\bk\in\mathbb{Z}^2_{\bm{0}}}  \int_\Gamma    (\bm{\sigma}^N_\bk\cdot\naby \dot{\eta}^N )\dot{\eta}^N \dy
 \dd W_s^\bk
\end{align*}
$\mathbb{P}$-a.s. for all $t\in I$, where the quadratic variation term obtained follows from \eqref{qv}.
Note that due to the divergence-free property of  $\bm{\sigma}^N_\bk$, 
\begin{align*}
\int_\Gamma    (\bm{\sigma}^N_\bk\cdot\naby \dot{\eta}^N )\dot{\eta}^N \dy=-\frac{1}{2}
\int_\Gamma    \mathrm{div}_\by(\bm{\sigma}^N_\bk)| \dot{\eta}^N|^2 \dy=0 .
\end{align*}
Hence, the noise term vanishes. Moreover, since $\eta_\star^N=P^N\eta_\star$ and $\eta_0^N=P^N \eta_0$, we have that
\begin{align*}
\frac{1}{2}&\Big( \epsilon_0\varrho_s\Vert \eta_\star^N\Vert_{L^2(\Gamma)}^2+\nu_e \Vert \eta_0^N\Vert_{L^2(\Gamma)}^2+\alpha \Vert \Dely\eta_0^N\Vert_{L^2(\Gamma)}^2+\beta \Vert \naby\eta_0^N\Vert_{L^2(\Gamma)}^2\Big)
\\&\leq
\frac{1}{2}\Big( \epsilon_0\varrho_s\Vert \eta_\star\Vert_{L^2(\Gamma)}^2+\nu_e \Vert \eta_0\Vert_{L^2(\Gamma)}^2+\alpha \Vert \Dely\eta_0\Vert_{L^2(\Gamma)}^2+\beta \Vert \naby\eta_0\Vert_{L^2(\Gamma)}^2\Big).
\end{align*}
Now, since $\bn$ is a unit vector, we also have by Young's inequality and the fact that  $g^N=P^Ng$ and $\mathbf{g}^N=P^N \mathbf{g}$
\begin{align*}
\int_0^t\int_\Gamma(g^N+\mathbf{g}^N\cdot\bn)\dot{\eta}^N\dy\ds
&\leq
\frac{\epsilon_0\varrho_s}{4}\sup_{t\in I}\Vert\dot{\eta}^N(t)\Vert_{L^2(\Gamma)}^2
+
\frac{2}{\epsilon_0\varrho_s}\int_0^t\left(\Vert g\Vert_{L^2(\Gamma)}^2+\Vert \mathbf{g}\Vert_{L^2(\Gamma)}^2 \right)\ds
\end{align*}
Finally, by integrating by part, we obtain
\begin{align*}
 \epsilon_0^2\varrho_s^2  \int_0^t\sum_{\bk\in\mathbb{Z}^2_{\bm{0}}}   \Vert\bm{\sigma}^N_\bk\cdot \naby\dot{\eta}^N \Vert_{L^2(\Gamma)}^2\ds 
&=
 \epsilon_0^2\varrho_s^2  \int_0^t\sum_{\bk\in\mathbb{Z}^2_{\bm{0}}}\int_\Gamma\bm{\sigma}^N_\bk\cdot \naby\dot{\eta}^N ( \overline{\bm{\sigma}}^N_\bk\cdot \naby\dot{\eta}^N )\dy\ds 
\\&=-
 \epsilon_0^2\varrho_s^2  \int_0^t\sum_{\bk\in\mathbb{Z}^2_{\bm{0}}}\int_\Gamma\dot{\eta}^N \,\bm{\sigma}^N_\bk\cdot \naby( \overline{\bm{\sigma}}^N_\bk\cdot \naby\dot{\eta}^N) \dy\ds 
  \\
  &=-
\gamma   \int_0^t \int_\Gamma\dot{\eta}^N \,\Dely\dot{\eta}^N \dy\ds 
  \\
  &=
\gamma  \int_0^t \Vert \naby\dot{\eta}^N \Vert_{L^2(\Gamma)}^2\ds.  
\end{align*}
Combining the above estimates, we conclude that $\mathbb{P}$-a.s., the inequality
\begin{equation}
\begin{aligned}
\label{energyGal}
\sup_{t\in I}\Big(\epsilon_0\varrho_s\Vert&\dot{\eta}^N(t)\Vert_{L^2(\Gamma)}^2+\nu_e \Vert \eta^N(t)\Vert_{L^2(\Gamma)}^2+\alpha \Vert \Dely\eta^N(t)\Vert_{L^2(\Gamma)}^2+\beta \Vert \naby\eta^N(t)\Vert_{L^2(\Gamma)}^2 \Big) 
 \lesssim \mathcal{E}(\mathrm{data})
\end{aligned}
\end{equation}
holds uniformly in $N\in\mathbb{N}$, where
\begin{align*}
\mathcal{E}(\mathrm{data}):=&\epsilon_0\varrho_s\Vert \eta_\star\Vert_{L^2(\Gamma)}^2+\nu_e \Vert \eta_0\Vert_{L^2(\Gamma)}^2+\alpha \Vert \Dely\eta_0\Vert_{L^2(\Gamma)}^2+\beta \Vert \naby\eta_0\Vert_{L^2(\Gamma)}^2
\\&+\frac{1}{\epsilon_0\varrho_s}\int_I\left(\Vert g\Vert_{L^2(\Gamma)}^2+\Vert \mathbf{g}\Vert_{L^2(\Gamma)}^2\right)\dt . 
\end{align*} 
Since the right-hand side of \eqref{energyGal} is finite by assumption, it follows that
\begin{align*}
&\eta^N \in L^{\infty} (I;W^{2,2}(\Gamma ))
,
\qquad
\dot{\eta}^N \in  L^\infty (I;L^{2}(\Gamma ))  \qquad \text{a.s..}
\end{align*} 
To improve this so that $\dot{\eta}^N$ is weakly continuous in time, we wish to apply
\cite[Theorem 1.8.5]{breit2018stochastically}. Thus, we need to  show that
\begin{align} \label{Kolmo}
\dot{\eta}^N \in  C^\kappa (I;W^{-2,2}(\Gamma ))  \qquad \text{a.s.}
\end{align} 
for all $\kappa\in (0,\frac{1}{4})$. To this end, we test \eqref{sde2} with any nonzero test function  $\zeta\in W^{2,2}(\Gamma)$ to obtain for any $t_0,t_1\in I$ with $[t_0, t_1] \subset I$,
\begin{equation}\label{eq:WC1}
\begin{aligned}
\epsilon_0\varrho_s\int_\Gamma (\dot{\eta}^N(t_1)- \dot{\eta}^N(t_0))\zeta\dy  
=&
\int_{t_0}^{t_1}\int_\Gamma 
( \beta \Delta_{\by}\eta^N   - \nu_e\eta^N +
 \mathbf{g}^N\cdot \bn+g^N)   \zeta\dy\dt
\\&+
\int_{t_0}^{t_1}\int_\Gamma 
( \gamma \dot{\eta}^N
-
 \alpha \Delta_{\by} \eta^N)  \Dely\zeta \dy\dt
\\&- 
\epsilon_0\varrho_s
\int_{t_0}^{t_1}\int_\Gamma
\sum_{\bk\in\mathbb{Z}^2_{\bm{0}}}   \dot{\eta}^N \,  \bm{\sigma}^N_\bk\cdot\naby \zeta
 \dy\dd W_t^\bk. 
\end{aligned}
\end{equation}
With the energy estimate \eqref{energyGal} in hand, the first term of \eqref{eq:WC1} satisfies
\begin{align*}
\mathbb{E}\bigg\vert& \int_{t_0}^{t_1}\int_\Gamma 
( \beta \Delta_{\by}\eta^N   - \nu_e\eta^N +
 \mathbf{g}^N\cdot \bn+g^N)   \zeta\dy\dt\bigg\vert^4
\\&\lesssim
 |t_1-t_0|^4\,\mathbb{E}\sup_{t\in I} \Vert\eta^N\Vert_{W^{2,2}(\Gamma)}^4 
 +
  |t_1-t_0|^2\,\mathbb{E}\int_{t_0}^{t_1}\big(\Vert\mathbf{g}^N\Vert_{L^{2}(\Gamma)}^4+\Vert g^N\Vert_{L^{2}(\Gamma)}^4\big)\dt
 \\&\lesssim
 |t_1-t_0|^2\, \mathcal{E}(\mathrm{data})^2 
\end{align*}
with a constant depending only on $\beta, \nu_e, T $ and $\Vert  \zeta \Vert_{L^2(\Gamma)}^4$. In the last step above, we have used the standard property of estimating a truncated Fourier series by its full series.
 Similarly
\begin{align*}
\mathbb{E}\bigg\vert \int_{t_0}^{t_1}\int_\Gamma 
( \gamma \dot{\eta}^N
-
 \alpha \Delta_{\by} \eta^N)  \Dely\zeta \dy\dt\bigg\vert^4
 &\lesssim
 |t_1-t_0|^4\,\mathbb{E}\sup_{t\in I}\big(\Vert\dot{\eta}^N\Vert_{L^{2}(\Gamma)}^4
 +\Vert\eta^N\Vert_{W^{2,2}(\Gamma)}^4 \big)
  \\&\lesssim
 |t_1-t_0|^4\,  \mathcal{E}(\mathrm{data})^2
\end{align*}
with a constant depending only on $\gamma, \alpha $ and $\Vert \Dely \zeta \Vert_{L^2(\Gamma)}^4$. For the stochastic integral, since $|e^{i\,\bk\cdot\by}|=1$, it follows from Burkholder--Davis--Gundy inequality that
\begin{align*}
\mathbb{E}\bigg\vert  \int_{t_0}^{t_1}&\sum_{\bk\in\mathbb{Z}^2_{\bm{0}}} 
\int_{\Gamma}\epsilon_0\varrho_s \dot{\eta}^N \bm{\sigma}^N_\bk\cdot\naby \zeta 
 \dy\dd W_t^\bk\bigg\vert^4
\\&\lesssim
\epsilon_0^4\varrho_s^4
\mathbb{E}\bigg(\int_{t_0}^{t_1} \sum_{\bk\in\mathbb{Z}^2_{\bm{0}}} 
\Big(\int_{\Gamma} \dot{\eta}^N \bm{\sigma}^N_\bk\cdot\naby \zeta 
 \dy\Big)^2\dt
 \bigg)^2
  \\
 &\lesssim 
 \gamma^2 \Big(\sum_{N\leq|\bk|\leq 2N}\frac{1}{|\bk|^2}\Big)^{-2}
 \mathbb{E}\bigg(\int_{t_0}^{t_1}  \sum_{\bk\in\mathbb{Z}^2_{\bm{0}}}
\Big(\int_{\Gamma} \dot{\eta}^N \frac{i\bk^\perp}{|\bk|^2}\bm{1}_{\{N\leq|\bk|\leq 2N\}}
e^{i\,\bk\cdot\by}\cdot\naby \zeta 
 \dy\Big)^2\dt\bigg)^2
  \\
 &\lesssim 
\Big(\sum_{N\leq|\bk|\leq 2N}\frac{1}{|\bk|^2}\Big)^{-2}
\Big(\sum_{N\leq|\bk|\leq 2N}\frac{1}{|\bk|^2}\Big)^{2}
 \mathbb{E}\bigg(\int_{t_0}^{t_1}   \Vert \dot{\eta}^N   \Vert_{L^2(\mathbb{T}^2)}^2\Vert
  \naby \zeta 
 \Vert_{L^2(\mathbb{T}^2)}^2\dt\bigg)^2
 \\
 &\lesssim  
|t_1-t_0|^2\,\mathbb{E}\sup_{t\in I}\Vert \dot{\eta}^N   \Vert_{L^2(\mathbb{T}^2)}^4
 \\&\lesssim
 |t_1-t_0|^2\,  \mathcal{E}(\mathrm{data})^2
\end{align*}
with a constant depending only on $\gamma $ and $\Vert \naby \zeta \Vert_{L^2(\Gamma)}^4$. If we now collect the three estimates above and observe that $|t_1-t_0|^4\leq T^2|t_1-t_0|^2$, we conclude that
\begin{align*}
\mathbb{E} \Vert   \dot{\eta}^N(t_1)- \dot{\eta}^N(t_0)  \Vert_{W^{-2,2}(\Gamma)}^4\lesssim |t_1-t_0|^2
\end{align*}
holds uniformly in $N$. Consequently, by the Kolmogorov continuity theorem, there exists a modification of $\dot{\eta}^N$ (not relabelled) such that \eqref{Kolmo} holds. This completes the proof of item (a) of Definition \ref{def:weakSolution}.
 
Pathwise uniqueness for \eqref{sde2} (or equivalently \eqref{sde1}) is straightforward since the system is linear. Indeed, if $\eta_1^N$ and $\eta_2^N$ are two solutions with the same data, then their difference $\eta_{12}^N=\eta_1^N-\eta_2^N$ satisfy the exact same equation but with zero initial conditions. Thus, $\eta_{12}^N$ satisfies the energy estimate \eqref{energyGal} with zero right-hand side leading to 
 \begin{align*}
\mathbb{E}\Big(\epsilon_0\varrho_s\Vert&\dot{\eta}_{12}^N(t)\Vert_{L^2(\Gamma)}^2+\nu_e \Vert \eta_{12}^N(t)\Vert_{L^2(\Gamma)}^2+\alpha \Vert \Dely\eta_{12}^N(t)\Vert_{L^2(\Gamma)}^2+\beta \Vert \naby\eta_{12}^N(t)\Vert_{L^2(\Gamma)}^2 \Big)
=0
\end{align*}
for any $t\in I$. Since the norms in the expectation are nonnegative, pathwise uniqueness immediately follows, i.e.,
\begin{align*}
\mathbb{P}\big(\omega\in\Sigma\,:\, \eta^N_1=\eta^N_2\big)=1.
\end{align*}


We can now proceed to show the second part of Theorem \ref{thm:main1} involving the passage to limit in
\begin{equation}
\begin{aligned}  \label{distributionalForm}
\epsilon_0\varrho_s\int_0^t\int_{\Gamma}\dot{\eta}^N\,\dot{\zeta}\dy\ds
=
&\int_0^t\int_{\Gamma}(\nu_e\eta^N\zeta   
+\alpha\Dely \eta^N\Dely \zeta + \beta \naby\eta^N\cdot\naby\zeta )\dy\ds 
\\
&-\int_0^t\int_{\Gamma}(g^N\,\zeta+\mathbf{g}^N\cdot\bn\zeta  
+ \gamma\,  \dot{\eta}^N\Dely\zeta 
)\dy\ds
\\&+\epsilon_0\varrho_s  
\int_0^t\int_{\Gamma} \sum_{\bk\in\mathbb{Z}^2_{\bm{0}}} \dot{\eta}^N \bm{\sigma}^N_\bk\cdot\naby \zeta 
 \dy\dd W_s^\bk
\end{aligned} 
\end{equation} 
$\mathbb{P}$-a.s.  for all $\zeta\in C^\infty_c(I\times \Gamma)$. Here, we recall that the  dataset $\mathbf{g}^N$ and $g^N$ are square-integrable  truncated Fourier series of the original dataset $(\mathbf{g},g)$ and so they automatically converges pointwise almost everywhere to their corresponding originals. Furthermore, given the energy estimate \eqref{energyGal}, there exists a subsequence (not relabelled) such that
\begin{align*}
\eta^N \overset{*}{\rightharpoonup} \eta \text{ in }& L^\infty(I;W^{2,2}(\Gamma)),
\\
\dot{\eta}_\star^N \overset{*}{\rightharpoonup}  \dot{\eta}_\star \text{ in }& L^\infty(I;L^{2}(\Gamma)),
\end{align*}
$\mathbb{P}$-a.s.. This is sufficient to pass to the limit in all the drift  terms in \eqref{distributionalForm} and obtain
\begin{align*}
f(\eta^N)\zeta\rightarrow f(\eta)\zeta  \qquad  \mathbb{P}-a.s. ,
\end{align*}
where
\begin{align*}
f(\eta^N)\zeta:=
\int_0^t\int_{\Gamma}(\epsilon_0\varrho_s\dot{\eta}^N\,\dot{\zeta}-\nu_e\eta^N\zeta   
-\alpha\Dely \eta^N\Dely \zeta - \beta \naby\eta^N\cdot\naby\zeta 
+g^N\,\zeta+\mathbf{g}^N\cdot\bn\zeta  
+ \gamma\,  \dot{\eta}^N\Dely\zeta 
)\dy\ds
\end{align*}
with an analogous definition for $f(\eta)\zeta$. Thus, it remains to pass to the limit in the stochastic integral.
By the Burkholder--Davis--Gundy inequality,  we deduce that 
\begin{align*}
\mathbb{E} \sup_{t\in I}&\bigg\vert\int_0^t\sum_{\bk\in\mathbb{Z}^2_{\bm{0}}} 
\int_{\Gamma}\epsilon_0\varrho_s  \dot{\eta}^N \bm{\sigma}^N_\bk\cdot\naby \zeta 
 \dy\dd W_s^\bk\bigg\vert 
\\&\lesssim
\epsilon_0\varrho_s 
\mathbb{E}\bigg(\int_{I}  \sum_{\bk\in\mathbb{Z}^2_{\bm{0}}} 
\Big(\int_{\Gamma} \dot{\eta}^N \bm{\sigma}^N_\bk\cdot\naby \zeta 
 \dy\Big)^2\dt\bigg)^{1/2} 
  \\
 &\lesssim 
 \sqrt{\gamma} \Big(\sum_{N\leq|\bk|\leq 2N}\frac{1}{|\bk|^2}\Big)^{-1/2}
 \mathbb{E}\bigg(\int_{I}  \sum_{\bk\in\mathbb{Z}^2_{\bm{0}}}
\Big(\int_{\Gamma} \dot{\eta}^N \frac{i\bk^\perp}{|\bk|^2}\bm{1}_{\{N\leq|\bk|\leq 2N\}}
e^{i\,\bk\cdot\by}\cdot\naby \zeta 
 \dy\Big)^2\dt\bigg)^{1/2}  
  \\
  & \lesssim \Big(\sum_{N\leq|\bk|\leq 2N}\frac{1}{|\bk|^2}\Big)^{-1/2}
\bigg( \sup_{N\leq|\bk|\leq 2N}\frac{1}{|\bk|^2}\bigg)^{1/2}    \mathbb{E}\bigg( \int_{I} \sum_{N\leq|\bk|\leq 2N}  \Big(\int_{\Gamma} (\dot{\eta}^N 
\naby \zeta)\,  e^{i\,\bk\cdot\by}
 \dy\Big)^2\dt    \bigg)^{1/2} 
  \\
 &\lesssim 
\Big(\sum_{N\leq|\bk|\leq 2N}\frac{1}{|\bk|^2}\Big)^{-1/2}
\bigg( \sup_{N\leq|\bk|\leq 2N}\frac{1}{|\bk|^2}\bigg)^{1/2}  
 \mathbb{E} \Big(\int_{I}  \Vert \dot{\eta}^N  
  \naby \zeta 
 \Vert_{L^2(\mathbb{T}^2)}^2\dt\Big)^{1/2}  
 \\
 &\lesssim  
\frac{1}{ \sqrt{2\pi\ln(2)} N }
\Vert \naby \zeta \Vert_{L^\infty(I\times \Gamma)} 
 \mathbb{E}\Big( \int_I  \Vert \dot{\eta}^N \Vert_{L^2(\Gamma)}^2 \dt  \Big)^{1/2} 
  \\
 &\lesssim 
\frac{1}{N}
  \mathcal{E}(\mathrm{data})^{1/2}    
\end{align*} 
with a constant depending only on $\gamma, T $ and $\Vert \naby \zeta \Vert_{L^\infty(I\times\Gamma)}$. In the last step, we have used H\"older inequality in time and the energy estimate \eqref{energyGal} whose right-hand side is uniform in $N$. The last estimate above, therefore, converges to zero as $N\rightarrow\infty$. Thus, we can conclude that the stochastic integral converges in law to zero. Since constants (in this case zero) have no randomness, this convergence in law implies convergence in probability for the full sequence on the same probability space. 
In fact, this convergence in law to zero also directly implies almost sure convergence for the full sequence and not for a subsequence as one would expect.
This is because if we denote the stochastic integral by $M_t^N$, then for any $\varepsilon>0$, we have by  Chebyshev's inequality, 
\begin{align*}
\sum_{N=0}^\infty\mathbb{P}(\vert M_t^N-0\vert>\varepsilon)
\leq \sum_{N=0}^\infty\frac{1}{\varepsilon^2}\mathbb{E}(\vert M_t^N \vert^2)
\leq \sum_{N=0}^\infty\frac{1}{\varepsilon^2}\mathbb{E}(\sup_{t\in I}\vert M_t^N \vert^2)
\lesssim \frac{1}{\varepsilon^2}\sum_{N=0}^\infty\frac{1}{ N^2}<\infty.
\end{align*}
This implies that the event $E^N:=\{\omega\in\Sigma\,:\,\vert M_t^N-0\vert>\varepsilon\}$ happens only finitely often (a.s.) and as such, by the first Borel--Cantelli lemma,
\begin{align*}
\mathbb{P}\Big(\limsup_{N\rightarrow\infty}E^N\Big)=
\mathbb{P}\Big(\bigcap_{N=1}^\infty\bigcup_{k=N}^\infty E^k\Big)=0
\end{align*}
for all $\varepsilon>0$. This completes the proof.
\end{proof}

\appendix
\section{ }
\begin{lemma}\label{lemA1}
Let $\epsilon_0\varrho_s,\gamma>0$,
let $\bm{\sigma}^N_\bk$ be given by \eqref{ansatz} and $W_t^\bk $ given by \eqref{BMimaginary}. Then for any $\xi\in C^2(\Gamma)$, the identity
\begin{align*}
\epsilon_0\varrho_s
\sum_{\bk\in\mathbb{Z}^2_{\bm{0}}} \bm{\sigma}^N_\bk  \cdot\naby\xi  \circ  \dd W_t^\bk 
=&  \epsilon_0\varrho_s
\sum_{\bk\in\mathbb{Z}^2_{\bm{0}}} \bm{\sigma}^N_\bk \cdot\naby\xi    \dd W_t^\bk 
          +
      \gamma
\Dely \xi   \dt
\end{align*}
holds.
\end{lemma}
\begin{proof}
Due to \eqref{qv} and the general Stratonovich-to-It\^o conversion rule
\begin{align*}
G_t\circ \dd H_t=G_t\dd H_t+\tfrac{1}{2}\dd[G,H]_t,
\end{align*}
it follows that the Stratonovich integral transforms into
\begin{align*}
\epsilon_0\varrho_s\ 
\sum_{\bk\in\mathbb{Z}^2_{\bm{0}}} \bm{\sigma}^N_\bk  \cdot\naby\dot{\eta}^N  \circ  \dd W_t^\bk 
=&  \epsilon_0\varrho_s
\sum_{\bk\in\mathbb{Z}^2_{\bm{0}}} \bm{\sigma}^N_\bk \cdot\naby\dot{\eta}^N    \dd W_t^\bk 
          +
         ( \epsilon_0\varrho_s\ )^2 
\sum_{\bk\in\mathbb{Z}^2_{\bm{0}}} \bm{\sigma}^N_\bk  \cdot\naby(\overline{\bm{\sigma}}^N_\bk \cdot\naby\dot{\eta}^N)    \dt
\\=&  \epsilon_0\varrho_s\theta^N
\sum_{N\leq|\bk|\leq 2N }  
          \tfrac{i\,\bk^\perp e^{ i\bk\cdot\by}}{|\bk|^2}\cdot\naby \dot{\eta}^N 
 \dd W_{t}^\bk
 \\&\qquad+
( \epsilon_0\varrho_s\theta^N)^2 \sum_{ N\leq|\bk|\leq 2N } 
           \tfrac{i\,\bk^\perp e^{i\bk\cdot\by}}{|\bk|^2}\cdot\naby\big( \tfrac{-i\,\bk^\perp e^{ -i\bk\cdot\by}}{|\bk|^2}\cdot\naby \dot{\eta}^N \big)
 \dt 
\end{align*}
where
\begin{align*}
\theta^N :=\frac{\sqrt{2\gamma}}{\epsilon_0\varrho_s}\Big(\sum_{N\leq|\bk|\leq 2N}\frac{1}{|\bk|^2}\Big)^{-1/2}.
\end{align*}
Now, observe that  
\begin{align*}
\tfrac{i\,\bk^\perp e^{ i\bk\cdot\by}}{|\bk|^2}&\cdot\naby\big( \tfrac{-i\,\bk^\perp e^{ -i\bk\cdot\by}}{|\bk|^2}\cdot\naby \dot{\eta}^N \big) 
\\&= 
\Big( \tfrac{-i\,k_2 e^{i\bk\cdot\by}}{|\bk|^2}\partial_1+\tfrac{i\,k_1 e^{i\bk\cdot\by}}{|\bk|^2}\partial_2\Big)\Big(\tfrac{i\,k_2 e^{-i\bk\cdot\by}}{|\bk|^2}\partial_1 \dot{\eta}^N- \tfrac{i\,k_1 e^{i\bk\cdot\by}}{|\bk|^2}\partial_2 \dot{\eta}^N\Big) 
  \\
&= 
  \tfrac{k_2^2  }{|\bk|^4} \partial_{11}^2 \dot{\eta}^N  
  -2
  \tfrac{k_1k_2  }{|\bk|^4} \partial_{12}^2 \dot{\eta}^N  
+
  \tfrac{k_1^2  }{|\bk|^4} \partial_{22}^2 \dot{\eta}^N  
\end{align*}
However, since points in $\mathbb{Z}^2_{\bm{0}}$ can be group into four symmetric points
\begin{align*}
(x,y), (-x,y), (x,-y), (-x,-y)
\end{align*}
whose grouping satisfies
\begin{align*}
xy+(-x)y+x(-y)+(-x)(-y)=0,
\end{align*}
it follows from symmetry that
\begin{align*}
\sum_{N\leq|\bk|\leq 2N}  \tfrac{k_1k_2  }{|\bk|^4} \partial_{12}^2 \dot{\eta}^N= \partial_{12}^2 \dot{\eta}^N\sum_{N\leq|\bk|\leq 2N} \tfrac{k_1k_2  }{|\bk|^4}  =0.
\end{align*}
On the other hand, since  any point $(x,y)\in\mathbb{Z}^2_{\bm{0}}$ can be paired with $(y,x)\in\mathbb{Z}^2_{\bm{0}}$
\begin{align*}
\sum_{N\leq|\bk|\leq 2N}  \tfrac{k_2^2  }{|\bk|^4}  =
\sum_{N\leq|\bk|\leq 2N}   \tfrac{k_1^2  }{|\bk|^4} 
=
\frac{1}{2}
\sum_{N\leq|\bk|\leq 2N}   \tfrac{k_1^2 +k_2^2 }{|\bk|^4}  
=
\frac{1}{2} 
\sum_{N\leq|\bk|\leq 2N}    \tfrac{1}{|\bk|^2} 
.
\end{align*} 
Therefore,
\begin{align*}
\sum_{N\leq|\bk|\leq 2N}   \tfrac{k_2^2  }{|\bk|^4}  \partial_{11}^2 \dot{\eta}^N  
+
\sum_{N\leq|\bk|\leq 2N}    \tfrac{k_1^2  }{|\bk|^4}  \partial_{22}^2 \dot{\eta}^N   
=
\frac{1}{2}
\Dely \dot{\eta}^N \sum_{N\leq|\bk|\leq 2N}    \tfrac{1}{|\bk|^2}  
\end{align*}
Putting all together finishes the proof.
\end{proof}

\section*{Statements and Declarations} 
\subsection*{Funding}
This work has been partly supported by Grant number 543675748  by the German Research
Foundation (DFG).
\subsection*{Author Contribution}
The authors wrote and reviewed the manuscript.
\subsection*{Conflict of Interest}
The authors declare that they have no conflict of interest.
\subsection*{Data Availability Statement}
Data sharing is not applicable to this article as no datasets were generated
or analyzed during the current study.
\subsection*{Competing Interests}
The authors have no competing interests to declare that are relevant to the content of this article.


\end{document}